\documentclass[a4j,12pt,showkeys]{article}

\usepackage{amsmath,amsthm,amssymb}
\usepackage{geometry}
\usepackage{hyperref}
\usepackage{cancel}
\usepackage{mathtools}
\usepackage{bm}

\date{}

\newtheorem{theo}{Theorem}[section]
\newtheorem{lemm}[theo]{Lemma}
\newtheorem{prop}[theo]{Proposition}
\newtheorem{cor}[theo]{Corollary}
\theoremstyle{definition}
\newtheorem{defi}[theo]{Definition}

\newtheorem{rem}[theo]{Remark}

\newcommand{\1}{\mbox{1}\hspace{-0.25em}\mbox{l}}

\providecommand{\keywords}[1]{\textbf{Keywords:} #1}

\allowdisplaybreaks[0]

\makeatletter

\@addtoreset{equation}{section}
\makeatother
\def\widebar{\accentset{{\cc@style\underline{\mskip10mu}}}}
\numberwithin{equation}{section}
\def\rnum#1{\expandafter{\romannumeral #1}} 
\def\Rnum#1{\uppercase\expandafter{\romannumeral #1}}

\title{Time-inconsistent consumption-investment problems in incomplete markets under general discount functions}

\author{Yushi Hamaguchi\thanks{Department of Mathematics, Kyoto University, Kyoto 606--8502, Japan, \href{mailto:hamaguchi@math.kyoto-u.ac.jp}{hamaguchi@math.kyoto-u.ac.jp}}}

\begin{document}
\maketitle


\begin{abstract}
In this paper, we study a time-inconsistent consumption-investment problem with random endowments in a possibly incomplete market under general discount functions. We provide a necessary condition and a verification theorem for an open-loop equilibrium consumption-investment pair in terms of a coupled forward-backward stochastic differential equation. Moreover, we prove the uniqueness of the open-loop equilibrium pair by showing that the original time-inconsistent problem is equivalent to an associated time-consistent one.
\end{abstract}

\keywords{Time-inconsistency; open-loop equilibrium; incomplete market; forward-backward stochastic differential equation.}


\section{Introduction}

In this paper, we investigate a time-inconsistent consumption-investment problem in an incomplete market under general discount functions. In recent years, time-inconsistent control problems have received remarkable attentions in control theory, mathematical finance and economics. Time-inconsistency for a dynamic control problem means that the so-called Bellman's principle of optimality does not hold. In other words, a restriction of an optimal control for a specific initial pair on a later time interval might not be optimal for that corresponding initial pair. Such a situation occurs for example in dynamic mean-variance control problems and in utility maximization problems for consumption-investment strategies under non-exponential discounting. In this paper, we focus on the later problem. In classical consumption-investment problems under discounted utility, the discount function which represents the time-preference of an investor is assumed to be exponential. This assumption implies that the discount rate is constant over time and provides the possibility to compare outcomes occurring at different times by discounting future utility at a constant rate. The compatibility of discounted utility at different times leads to time-consistency of the problem, and hence we can use the classical dynamic programming approaches and the analytical tools of the Hamilton--Jacobi--Bellman (HJB) equations to obtain the optimal strategy together with the value function. However, results from experimental studies indicate that discount rates for near future are much lower than discount rates for the time further away in future, that contradict the assumption of exponential discounting; see e.g.\ Ainslie~\cite{a_Ainslie_75}. Therefore, it is important to investigate consumption-investment problems under non-exponential discounting. Unfortunately, in the case of non-exponential discounting, we cannot compare discounted utility at different times, and hence the problem becomes time-inconsistent. In order to handle that problem in a time-consisting way, we must introduce another concept of solutions instead of an optimal control.

In the literature of time-inconsistent control problems, several concepts of time-consistent solutions have been introduced and investigated. The main approaches to handle time-inconsistent control problems are to seek for, instead of optimal controls, time-consistent equilibrium controls, which are within a game theoretic framework. Bj\"{o}rk, Khapko and Murgoci~\cite{a_Bjork-_17} introduced an equilibrium strategy in a Markovian setting, and provided an extended HJB equation together with a verification theorem. Yong~\cite{a_Yong_12} considered a multi-players differential game framework with a hierarchical structure, and derived the so-called equilibrium HJB equation. The solution concepts considered in \cite{a_Bjork-_17,a_Yong_12} were closed-loop equilibrium strategies, which are equilibria for ``decision rules'' that a controller uses to select a ``control action'' based on each state. Mathematically, a closed-loop strategy is a mapping from states to control actions, which is chosen independently of initial conditions. The methods of \cite{a_Bjork-_17,a_Yong_12} to treat time-inconsistent control problems are extensions of the classical dynamic programming approaches. In contrast, Hu, Jin and Zhou~\cite{a_Hu-Jin-Zhou_12,a_Hu-Jin-Zhou_17} defined an open-loop equilibrium control, which is an equilibrium concept for a ``control process'' that a controller chooses based on the initial condition. In a time-inconsistent linear-quadratic stochastic control problem, they used a duality method in the spirit of the classical maximum principle, and characterized an open-loop equilibrium control by a ``flow'' of forward-backward stochastic differential equations (FBSDEs for short), which is a coupled system consisting of a single stochastic differential equation (SDE) and a continuum of backward SDEs (BSDEs) defined on different time intervals. The solvability of a flow of FBSDEs remains a challenging open problem except for some special cases; see the author's work~\cite{a_Hamaguchi_20} for small-time solvability of a flow of FBSDEs with Lipschitz continuous coefficients. In order to handle time-inconsistent consumption-investment problems in a continuous-time model, Ekeland and Pirvu~\cite{a_Ekeland-Pirvu_08} were the first to provide a precise definition of the equilibrium concept within a class of closed-loop strategies in a Markovian model. They characterized the equilibrium policy through the solution of a flow of BSDEs, and they showed, with a special form of the discount function, this flow of BSDEs has a solution. Zhao, Shen and Wei~\cite{a_Zhao-_14} studied a time-inconsistent consumption-investment problem in a non-Markovian model with the logarithmic utility function and a general (Lipschitz continuous) discount function. They adopted the multi-players differential game approach introduced by \cite{a_Yong_12}, and obtained a time-consistent strategy. Zhao, Wang and Wei~\cite{a_Zhao-_16} also adopted the multi-players differential game approach to investigate a time-inconsistent consumption-investment-reinsurance problem for an insurer with the exponential utility function and a general discount function (with some structural assumptions). Their model is non-Markov, while the interest rate is assumed to be deterministic. In contrast, within the class of open-loop controls, Alia et al.~\cite{a_Alia-_17} studied a time-inconsistent consumption-investment problem under a general utility function and a general (Lipschitz continuous) discount function. They derived, by using the same duality method as \cite{a_Hu-Jin-Zhou_12,a_Hu-Jin-Zhou_17}, a flow of FBSDEs which characterizes the open-loop equilibrium consumption-investment pair. They assumed that the market is complete and the interest rate is deterministic. We mention that their assumptions imposed on the utility function are too strong to apply to the exponential utility case; see Remark~\ref{remark: Alia} for some comments on their results. In fact, most previous researches assumed that the interest rate is deterministic, or that the utility function has a specific form or satisfies some strong assumptions.

The aim of this paper is to investigate an open-loop equilibrium strategy pair of a time-inconsistent consumption-investment problem under general utility functions and general discount functions. The novelties of this paper are as follows:
\begin{enumerate}
\renewcommand{\labelenumi}{(\roman{enumi})}
\item
The market is possibly incomplete, and the interest rate is allowed to be a stochastic process. Moreover, the investor is assumed to be endowed with a random income and a random terminal lump-sum payment.
\item
We provide a necessary condition (Theorem~\ref{theorem: necessary condition}) and a verification theorem (Theorem~\ref{theorem: verification}) for an open-loop equilibrium pair. Their conditions are related to the solvability of the corresponding fully coupled FBSDE, which is more tractable than a flow of FBSDEs appearing in \cite{a_Ekeland-Pirvu_08,a_Alia-_17}.
\item
By using the above results, we prove that finding an open-loop equilibrium pair of the time-inconsistent consumption-investment problem is equivalent to finding an optimal pair of a time-consistent one (Theorem~\ref{theorem: equivalent problems}).
\item
As a consequence, we obtain the uniqueness of an open-loop equilibrium pair satisfying suitable integrability conditions.
\end{enumerate}
It is worth to mention that only a few papers of time-inconsistent problems have studies the uniqueness of the equilibrium. To the best of our knowledge, in the closed-loop framework, there have not been any positive results on the uniqueness of equilibrium strategies. In a time-inconsistent linear-quadratic stochastic control problem, Hu--Jin--Zhou~\cite{a_Hu-Jin-Zhou_17} showed the uniqueness of the open-loop equilibrium control when the state is one dimensional and coefficients are deterministic. In contrast, we show the uniqueness of the open-loop equilibrium pair of a time-inconsistent consumption-investment problem with random coefficients. Let us remark on the result (\rnum{3}). Our result is quit different from the result in \cite{a_Bjork-_17} in which in order to formulate the equivalent time-consistent problem one needs to know the (closed-loop) equilibrium strategy. In contrast to the above paper, the equivalent time-consistent problem in Theorem~\ref{theorem: equivalent problems} does not depend on the underlying equilibrium pair. Thus, our result is practically important since it suffices to find an optimal pair of the standard time-consistent problem in order to obtain an open-loop equilibrium pair of the original time-inconsistent problem. Recently, Alia~\cite{a_Alia_19} studied a time-inconsistent control problem for a jump diffusion model under a general discount function, and constructed an equivalent time-consistent control problem. He firstly proved the equivalence of two problems, and then characterized an open-loop equilibrium control of the original time-inconsistent problem by solving the associated time-consistent one. However, since our model specified below does not satisfy these conditions, we cannot use his result directly. In contrast to the above mentioned paper, we firstly characterize an open-loop equilibrium pair, and then, as a consequence, we get the associated time-consistent problem. Our methods to prove the main theorems (Theorems~\ref{theorem: necessary condition},\,\ref{theorem: verification}) are inspired by Horst et al.~\cite{a_Horst-_14}. They studied a time-consistent utility maximization problem for the terminal wealth (without consumption) in an incomplete market with a general utility function, and characterized an optimal control by a fully coupled FBSDE, which is different from that appearing in the classical duality method. The key observation of their method is to derive the dynamics of the density process of an equivalent martingale measure under which the optimal wealth process becomes a true martingale. This observation is called the martingale optimality principle, which goes back to Hu, Imkeller and M\"{u}ller~\cite{a_Hu-Imkeller-Muller_05} who treated some particular utility functions, that is, exponential, logarithmic and power utility functions. See also Cheridito and Hu~\cite{a_cheridito-Hu_11} for the martingale optimality principle for (time-consistent) consumption-investment problems under non-convex constraints, where the above three types of utility functions and the exponential discount function were treated. In this paper, we consider general utility functions defined on the whole real line and general discount functions. Unfortunately, due to the time-inconsistency, the martingale optimality principle does not make sense in our problem. Our idea to tackle a time-inconsistent consumption-investment problem is, loosely speaking, to combine the technique of \cite{a_Horst-_14} and the duality method for time-inconsistent control problems. In such a way we obtain a characterization of an open-loop equilibrium pair by an FBSDE, which has more information about the structure of an open-loop equilibrium pair than the characterization by the duality method; see Remark~\ref{remark: maximum principle}.

The remainder of this paper is organized as follows: In Section~\ref{section: model}, we introduce our financial market model. In Section~\ref{section: necessary condition}, we provide a necessary condition for an open-loop equilibrium pair in terms of an FBSDE. As a converse result, in Section~\ref{section: verification theorem}, we derive a verification theorem, that is, we show that a solution of the FBSDE appearing in the necessary condition allows to construct an open-loop equilibrium pair. By using these results, in Section~\ref{section: equivalent problem}, we relate our time-inconsistent problem to a time-consistent one, and show the uniqueness of an open-loop equilibrium pair of the original time-inconsistent problem. Some proofs of technical lemmas are collected in Appendix~\ref{appendix}.


\section{The model}\label{section: model}

Let $T>0$ be a finite time horizon and let $W=(W_t)_{t\in[0,T]}$ be a $d$-dimensional Brownian motion on a complete probability space $(\Omega,\mathcal{F},\mathbb{P})$. $\mathbb{F}=(\mathcal{F}_t)_{t\in[0,T]}$ denotes the $\mathbb{P}$-augmentation of the filtration generated by $W$. Denote by $\mathbb{E}_t[\cdot]$ the conditional expectation given by $\mathcal{F}_t$ for each $t\in[0,T)$. $\1_A$ denotes the indicator function for a set $A$, and $\mathrm{Leb}$ denotes the Lebesgue measure on $\mathbb{R}$. For $t\in[0,T]$, $p,q\geq1$ and $\mathbb{H}=\mathbb{R},\mathbb{R}^d$, $L^\infty_{\mathcal{F}_t}(\Omega;\mathbb{H})$ denotes the set of all $\mathbb{H}$-valued $\mathcal{F}_t$-measurable random variables, $L^p_\mathbb{F}(\Omega;L^q(t,T;\mathbb{H}))$ denotes the set of all $\mathbb{H}$-valued predictable processes $X$ such that $\mathbb{E}[(\int^T_t|X_s|^q\,ds)^{p/q}]<\infty$, and $L^p_\mathbb{F}(\Omega;C([t,T];\mathbb{H}))$ denotes the set of all $\mathbb{H}$-valued adapted and continuous processes $X$ such that $\mathbb{E}[\sup_{s\in[t,T]}|X_s|^p]<\infty$. We also define $L^p_\mathbb{F}(t,T;\mathbb{H}):=L^p_\mathbb{F}(\Omega;L^p(t,T;\mathbb{H}))$.

We consider a financial market consisting of a riskless asset $S^0$ and $d$ risky assets $\tilde{S}^i,\ i=1,\dots,d$. The prices of these assets follow the dynamics
\begin{equation}\label{S^0 dynamics}
	\begin{cases}
		dS^0_t=r_tS^0_t\,dt,\ t\in[0,T],\\
		S^0_0=1,
	\end{cases}
\end{equation}
and
\begin{equation*}
	\begin{cases}
		d\tilde{S}^i_t=\tilde{S}^i_t(dW^i_t+b^i_t\,dt),\ t\in[0,T],\\
		\tilde{S}^i_0=\tilde{s}^i_0>0,
	\end{cases}
	i=1,\dots,d,
\end{equation*}
where $r$ and $b^i,\ i=1,\dots,d,$ are $\mathbb{R}$-valued predictable processes. We assume that the interest rate process $r$ and the excess rate of return vector process $\theta:=(b^1-r,\dots,b^d-r)^\top$ are bounded. It is well-known that, in this market model, arbitrage opportunities are excluded; see the textbook~\cite{b_Karatzas-Shreve_98}. Consider a small investor receiving an income $e_s$ at each intermediate time $s\in[0,T]$ and a lump-sum payment $E$ at time $T$, who can consume at intermediate times and invest in the financial market. Throughout this paper, we assume the following:
\begin{enumerate}
\renewcommand{\labelenumi}{(\roman{enumi})}
\item
$r$ is an $\mathbb{R}$-valued bounded predictable process, and $\theta$ is an $\mathbb{R}^d$-valued bounded predictable process.
\item
$e$ is an $\mathbb{R}$-valued predictable process such that $\int^T_0|e_s|\,ds<\infty$ a.s., and $E$ is an $\mathbb{R}$-valued $\mathcal{F}_T$-measurable random variable.
\end{enumerate}
Let $d_1\in\{1,\dots,d\}$ be fixed, and assume that the investor can invest in the riskless asset $S^0$ and the risky assets $\tilde{S}^1,\dots,\tilde{S}^{d_1}$, while the assets $\tilde{S}^{d_1+1},\dots,\tilde{S}^d$ cannot be invested into. Note that, if $d_1<d$ (resp.\,$d_1=d$), then the market is incomplete (resp.\,complete). Define $W^\mathcal{H}:=(W^1,\dots,W^{d_1},0,\dots,0)^\top$, $W^\mathcal{O}:=(0,\dots,0,W^{d_1+1},\dots,W^d)^\top$, $\theta^\mathcal{H}:=(\theta^1,\dots,\theta^{d_1},0,\dots,0)^\top$, and $S^\mathcal{H}:=\int^\cdot_0(dW^\mathcal{H}_s+\theta^\mathcal{H}_s\,ds)$. Here, the notation $\mathcal{H}$ refers to ``hedgeable'' and $\mathcal{O}$ to ``orthogonal''. (We borrowed these notations from \cite{a_Horst-_14}.) Hereafter, for each $x=(x^1,\dots,x^d)^\top\in\mathbb{R}^d$, we use the notations
\begin{equation*}
	x^\mathcal{H}=(x^1,\dots,x^{d_1},0,\dots,0)^\top\ \text{and}\ x^\mathcal{O}=(0,\dots,0,x^{d_1+1},\dots,x^d)^\top.
\end{equation*}
If the investor whose initial wealth at time $t\in[0,T)$ is $x_t\in\mathbb{R}$ consumes at a predictable rate $c$ and invests according to an $\mathbb{R}^d$-valued predictable trading strategy $\pi=(\pi^1,\dots,\pi^{d_1},0,\dots,0)^\top$, where $\pi^i_s$ is the amount of money invested in stock $i$ at time $s$, then his/her wealth $X^{(c,\pi)}=X^{(c,\pi,t,x_t)}$ evolves as
\begin{equation}\label{wealth SDE}
	X^{(c,\pi)}_s=x_t+\int^s_tr_uX^{(c,\pi)}_u\,du+\int^s_t\pi_u\cdot dS^\mathcal{H}_u+\int^s_t(e_u-c_u)\,du,\ s\in[t,T].
\end{equation}
If $\int^T_t(|c_s|+|\pi_s|^2)\,ds<\infty$ a.s., then SDE~(\ref{wealth SDE}) has a unique continuous solution $X^{(c,\pi)}$.

In this paper, we consider a utility maximization problem for a consumption-investment pair $(c,\pi)$ under general discount functions. In order to define the reward functional, we introduce a class $\mathbb{U}$ of utility functions and a class $\Lambda$ of discount functions. We say that a function $U:\mathbb{R}\to\mathbb{R}$ is in $\mathbb{U}$ if $U$ is three times differentiable, strictly increasing, strictly concave, and satisfies the Inada condition: $\lim_{x\to\infty}U'(x)=0$, $\lim_{x\to-\infty}U'(x)=\infty$. Also, with the notation $\Delta[0,T]:=\{(t,s)\,|\,0\leq t\leq s\leq T\}$, we say that a function $\lambda:\Delta[0,T]\to\mathbb{R}_+$ is in $\Lambda$ if $\lambda$ is continuous, strictly positive and satisfies $\lambda(t,t)=1$ for any $t\in[0,T]$. For example, the exponential utility function is in the class $\mathbb{U}$. Some possible examples of discount functions in the class $\Lambda$ are as follows:
\begin{itemize}
\item
Exponential discounting: $\lambda(t,s)=e^{-\delta(s-t)}$ with some constant $\delta\geq0$;
\item
Convex combination of two exponential discounting: $\lambda(t,s)=\alpha e^{-\delta(s-t)}+(1-\alpha)e^{-\gamma(s-t)}$ with some constants $\alpha\in(0,1)$ and $\delta,\,\gamma>0$ such that $\delta\neq\gamma$;
\item
Quasi-exponential discounting: $\lambda(t,s)=(1+\alpha(s-t))e^{-\delta(s-t)}$ with some constants $\alpha,\,\delta>0$;
\item
Hyperbolic discounting: $\lambda(t,s)=\frac{1}{1+\delta(s-t)}$ with some constant $\delta>0$.
\end{itemize}
In the above examples, $\lambda$ can be seen as a function of $s-t$. More generally, we can consider the following discount function:
\begin{itemize}
\item
Exponential discounting with reference-time-dependent discount rate: $\lambda(t,s)\\=e^{-\delta(t)(s-t)}$ with some nonnegative continuous function $\delta:[0,T]\to\mathbb{R}_+$.
\end{itemize}
Note that the discount function in the last example cannot be written as a function of $s-t$. Therefore, it is beyond the class of discount functions considered in \cite{a_Ekeland-Pirvu_08,a_Zhao-_14,a_Zhao-_16,a_Alia-_17}.

Suppose that $U_1,\,U_2\in\mathbb{U}$ and $\lambda_1,\,\lambda_2\in\Lambda$ are given. For each initial condition $(t,x_t)\in[0,T)\times\mathbb{R}$, we impose the following condition on a consumption-investment pair $(c,\pi)$:
\begin{itemize}
\item[(H0)$_{(t,x_t)}$]
$c$ is an $\mathbb{R}$-valued predictable process, $\pi$ is an $\mathbb{R}^d$-valued predictable process such that $\pi=(\pi^1,\dots,\pi^{d_1},0,\dots,0)^\top$, and it holds that
\begin{equation*}
	\int^T_t(|c_s|+|\pi_s|^2)\,ds<\infty\ \text{a.s.\ and}\ \mathbb{E}\left[\int^T_t|U_1(c_s)|\,ds+|U_2(X^{(c,\pi,t,x_t)}_T+E)|\right]<\infty.
\end{equation*}
\end{itemize}

We simply denote by (H0)$_x$ the condition (H0)$_{(0,x)}$ for each $x\in\mathbb{R}$.


\begin{rem}
The condition (H0)$_{(t,x_t)}$ depends on the utility functions $U_1,\,U_2\in\mathbb{U}$, while it does not depend on the discount functions $\lambda_1,\,\lambda_2\in\Lambda$. The same is true for (H1)$_{x,p}$ and (H2)$_{x,p}$ defined below.
\end{rem}

Denote by $\Pi^{(t,x_t)}_0$ (resp.\ $\Pi^x_0$) the set of pairs $(c,\pi)$ satisfying (H0)$_{(t,x_t)}$ (resp.\ (H0)$_x$). Suppose that the investor seeks for a pair $(c,\pi)\in\Pi^{(t,x_t)}_0$ that maximizes the reward functional
\begin{equation}\label{reward functional}
	R(c,\pi;t,x_t):=\mathbb{E}_t\left[\int^T_t\lambda_1(t,s)U_1(c_s)\,ds+\lambda_2(t,T)U_2(X^{(c,\pi)}_T+E)\right].
\end{equation}
When $\lambda_1,\,\lambda_2$ are exponential discount functions (with the same discount rate), it is well-known that the maximization problem for reward functional (\ref{reward functional}) is time-consistent. However, if we assume that $\lambda_1,\,\lambda_2$ are non-exponential discount functions, then it is time-inconsistent in general; see Yong~\cite{a_Yong_12}. Instead of finding a global optimal pair (which does not exist), we seek for an open-loop equilibrium pair $(c^*,\pi^*)$.
\par
To define an open-loop equilibrium pair, we introduce the set of perturbations; For each $t\in[0,T)$, define
\begin{equation*}
	\chi_t:=\{(\kappa,\eta)\,|\,\kappa\in L^\infty_{\mathcal{F}_t}(\Omega;\mathbb{R}),\ \eta=(\eta^1,\dots,\eta^{d_1},0,\dots,0)^\top\in L^\infty_{\mathcal{F}_t}(\Omega;\mathbb{R}^d)\}.
\end{equation*}


\begin{defi}
Let $U_1,\,U_2\in\mathbb{U}$ and $\lambda_1,\,\lambda_2\in\Lambda$ be given. For a given initial wealth $x\in\mathbb{R}$, we call $(c^*,\pi^*)\in\Pi^x_0$ an open-loop equilibrium pair if, for any $t\in[0,T)$ and $(\kappa,\eta)\in\chi_t$, it holds that
\begin{equation*}
	\limsup_{\epsilon\downarrow0}\frac{R(c^{t,\epsilon},\pi^{t,\epsilon};t,X^*_t)-R(c^*,\pi^*;t,X^*_t)}{\epsilon}\leq0\ \text{a.s.},
\end{equation*}
where $X^*=X^{(c^*,\pi^*,0,x)}$, and the pair $(c^{t,\epsilon},\pi^{t,\epsilon})=(c^{t,\epsilon,\kappa},\pi^{t,\epsilon,\eta})$ is defined by
\begin{equation*}
	\begin{cases}
		c^{t,\epsilon}_s:=c^*_s+\kappa\1_{[t,t+\epsilon)}(s),\\
		\pi^{t,\epsilon}_s:=\pi^*_s+\eta\1_{[t,t+\epsilon)}(s),
	\end{cases}
\end{equation*}
for $s\in[t,T]$.
\end{defi}


\begin{rem}
The above definition of an open-loop equilibrium pair is inspired by \cite{a_Hu-Jin-Zhou_12,a_Hu-Jin-Zhou_17}. An open-loop equilibrium pair is a time-consistent consumption-investment strategy pair satisfying a kind of local optimality condition. Note that we consider only bounded perturbations $(\kappa,\eta)\in\chi_t$ of a pair $(c^*,\pi^*)\in\Pi^x_0$.
\end{rem}


\begin{rem}\label{remark: Alia}
If the interest rate process $r$ is deterministic, $d_1=d$ (i.e.\,the market is complete), $(e,E)=(0,0)$ (i.e.\,without endowments), $U^{(3)}_1,\,U^{(3)}_2$ are bounded, and $\lambda_1,\,\lambda_2$ are of the forms $\lambda_1(t,s)=\lambda(s-t)$ and $\lambda_2(t,T)=\lambda(T-t)$ with some Lipschitz continuous function $\lambda$, then our model becomes (essentially) the same one as in Alia et al.~\cite{a_Alia-_17}. This paper goes beyond \cite{a_Alia-_17} for this reason. Essentially, some boundedness conditions of $U_1$ and $U_2$ are weakened in the present paper. This is important since our setting can be applied to the exponential utility function and the class of utility functions introduced by Fromm and Imkeller~\cite{a_Fromm-Imkeller_20}.
\end{rem}

The aim of this paper is to characterize an open-loop equilibrium pair by an FBSDE. We must treat the integrability conditions and the limit operations carefully to overcome the technical difficulties arising in the literature. To do so, let us introduce further conditions of consumption-investment pairs. Let $U_1,\,U_2\in\mathbb{U}$, $x\in\mathbb{R}$ and $p>1$ be given.
\begin{itemize}
\item[(H1)$_{x,p}$]
$(c,\pi)$ satisfies (H0)$_x$ and
\begin{equation*}
	\mathbb{E}\left[\int^T_0U'_1(c_s)^p\,ds+U'_2(X^{(c,\pi)}_T+E)^p\right]<\infty.
\end{equation*}
\item[(H2)$_{x,p}$]
$(c,\pi)$ satisfies (H1)$_{x,p}$. Moreover, there exists a constant $q>1$ such that:
\begin{enumerate}
\renewcommand{\labelenumi}{(\roman{enumi})}
\item
$\mathbb{E}[\int^T_0M_1(c_s;\delta)^q\,ds]<\infty$ for any $\delta\geq0$;
\item
For any $t\in[0,T)$ and $(\kappa,\eta)\in\chi_t$, the family of $\mathcal{F}_T$-measurable random variables $\{M_2(X^{(c,\pi)}_T+E;|\xi^{t,\epsilon}_T|)^q\}_{\epsilon\in(0,T-t)}$ is uniformly integrable.
\end{enumerate}
\end{itemize}
Here we used the notations
\begin{equation}\label{definition of M}
	M_i(x;\delta):=\max_{y\in\mathbb{R},\,|y|\leq\delta}|U''_i(x+y)|,\ x\in\mathbb{R},\ \delta\geq0,\ i=1,2,
\end{equation}
and, for each $t\in[0,T)$, $(\kappa,\eta)\in\chi_t$ and $\epsilon\in(0,T-t)$, $\xi^{t,\epsilon}=\xi^{t,\epsilon,\kappa,\eta}$ is defined as the unique solution of the SDE
\begin{equation}\label{perturbation SDE}
	\begin{cases}
		d\xi^{t,\epsilon}_s=r_s\xi^{t,\epsilon}_s\,ds+\eta\1_{[t,t+\epsilon)}(s)\cdot dS^\mathcal{H}_s-\kappa\1_{[t,t+\epsilon)}(s)ds,\ s\in[t,T],\\
		\xi^{t,\epsilon}_t=0.
	\end{cases}
\end{equation}
For $i=1,2$, we denote by $\Pi^{x,p}_i$ the set of $(c,\pi)$ satisfying (H$i$)$_{x,p}$. Clearly, it holds that $\Pi^{x,p}_2\subset\Pi^{x,p}_1\subset\Pi^x_0$ for $x\in\mathbb{R}$ and $p>1$, and $\Pi^{x,p}_i\subset\Pi^{x,q}_i$ for $1<q<p$, $x\in\mathbb{R}$ and $i=1,2$. (H2)$_{x,p}$ is a technical condition which will be used in Lemma~\ref{lemma: limit2} below. Now let us observe some fundamental properties of SDE~\eqref{perturbation SDE}. 


\begin{lemm}\label{lemma: perturbation SDE}
Let $t\in[0,T)$ and $(\kappa,\eta)\in\chi_t$ be fixed.
\begin{enumerate}
\renewcommand{\labelenumi}{(\roman{enumi})}
\item
For any $\epsilon\in(0,T-t)$, SDE~\eqref{perturbation SDE} has a unique solution $\xi^{t,\epsilon}\in\bigcap_{\gamma\geq1}L^\gamma_\mathbb{F}(\Omega;C([t,T];\mathbb{R}))$.
\item
For any $\gamma\geq1$, there exists a constant $C_\gamma=C(\gamma,T,\|r\|_\infty,\|\theta\|_\infty)>0$ such that, for any $\epsilon\in(0,T-t)$, it holds that
\begin{equation*}
	\mathbb{E}_t\Bigl[\sup_{t\leq s\leq T}|\xi^{t,\epsilon}_s|^{2\gamma}\Bigr]\leq\bigl(C_\gamma\epsilon(|\kappa|^2+|\eta|^2)\bigr)^\gamma\ \text{a.s.}
\end{equation*}
\item
For any $c>0$, it holds that
\begin{equation*}
	\sup_{\epsilon\in(0,T-t)}\mathbb{E}\left[\exp\left(c|\xi^{t,\epsilon}_T|\right)\right]<\infty.
\end{equation*}
\end{enumerate}
\end{lemm}


\begin{proof}
See Appendix~\ref{appendix}.
\end{proof}

The following lemma says that, for a large class of utility functions, (H1)$_{x,p}$ and (H2)$_{x,p}$ become equivalent.


\begin{lemm}\label{lemma: utility functions}
Let $U_1,\,U_2\in\mathbb{U}$ be given. Suppose that $U_i$, $i=1,2$, satisfy the following conditions:
\begin{enumerate}
\renewcommand{\labelenumi}{(\roman{enumi})}
\item
$\frac{U''_i}{U'_i}$ is bounded;
\item
There exists a constant $K>0$ such that $\frac{U'_i(x)}{U'_i(y)}\leq\exp(K(y-x))$ for any $x,y\in\mathbb{R}$ with $x\leq y$. 
\end{enumerate}
Then $\Pi^{x,p}_1=\Pi^{x,p}_2$ for any $x\in\mathbb{R}$ and $p>1$.
\end{lemm}


\begin{proof}
See Appendix~\ref{appendix}.
\end{proof}


\begin{rem}
Clearly the exponential utility function $U(x)=-\exp(-\gamma x)$ with some constant $\gamma>0$ satisfies the assumptions of Lemma~\ref{lemma: utility functions}. More generally, consider the utility function
\begin{equation*}
	U(x)=-\int^\infty_x\Bigl(\int^\infty_y\exp(-\kappa(z))\,dz\Bigr)\,dy
\end{equation*}
with some twice differentiable function $\kappa:\mathbb{R}\to\mathbb{R}$ satisfying
\begin{align*}
	0<\inf_{x\in\mathbb{R}}\kappa'(x)\leq\sup_{x\in\mathbb{R}}\kappa'(x)<\infty\ \text{and}\ 0\leq\inf_{x\in\mathbb{R}}\kappa''(x)\leq\sup_{x\in\mathbb{R}}\kappa''(x)<\infty,
\end{align*}
which was introduced by Fromm and Imkeller~\cite{a_Fromm-Imkeller_20}. The above utility function also satisfies the assumptions of Lemma~\ref{lemma: utility functions}. Indeed, by Lemma~1.2 in \cite{a_Fromm-Imkeller_20}, $\frac{U''}{U'}$ is bounded. Moreover, for any $x,y\in\mathbb{R}$ with $x\leq y$,
\begin{align*}
	U'(x)&=\int^\infty_x\exp(-\kappa(z))\,dz=\int^\infty_y\exp\bigl(-\kappa(z-(y-x))\bigr)\,dz\\
	&\leq\int^\infty_y\exp\bigl(-\kappa(z)+\kappa'(z)(y-x)\bigr)\,dz\\
	&\leq\int^\infty_y\exp(-\kappa(z))\,dz\,\exp\bigl(\|\kappa'\|_\infty(y-x)\bigr)\\
	&=U'(y)\exp\bigl(\|\kappa'\|_\infty(y-x)\bigr),
\end{align*}
where in the first inequality we used the convexity of $\kappa$. Therefore the second assumption of Lemma~\ref{lemma: utility functions} is also satisfied.
\end{rem}


\section{A necessary condition for an equilibrium pair}\label{section: necessary condition}

\paragraph{}
\ \,In this section, we provide a necessary condition for an open-loop equilibrium pair. Throughout this section, we fix $U_1,U_2\in\mathbb{U}$ and $\lambda_1,\lambda_2\in\Lambda$.


\begin{theo}\label{theorem: necessary condition}
Fix an initial wealth $x\in\mathbb{R}$, and let $(c^*,\pi^*)\in\Pi^{x,p}_2$ for some $p>1$. If $(c^*,\pi^*)$ is an open-loop equilibrium pair, then there exists a pair $(Y,Z)$ such that:
\begin{enumerate}
\renewcommand{\labelenumi}{(\roman{enumi})}
\item
$Y$ is an $\mathbb{R}$-valued continuous adapted process such that $U'_2(X^*+Y)$ is in $L^p_\mathbb{F}(\Omega;C([0,T];\mathbb{R}))$, and $Z$ is an $\mathbb{R}^d$-valued predictable process satisfying $\int^T_0|Z_s|^2\,ds<\infty$ a.s.;
\item
$(Y,Z)$ satisfies the following BSDE:
\begin{equation}\label{BSDE for (Y,Z)}
	\begin{cases}
		dY_s=Z_s\cdot dW_s+f^*(s,Y_s,Z_s)\,ds,\ s\in[0,T],\\
		Y_T=E,
	\end{cases}
\end{equation}
where the generator $f^*: \Omega\times[0,T]\times\mathbb{R}\times\mathbb{R}^d\to\mathbb{R}$ is defined by
\begin{equation*}
	f^*(s,y,z):=-\frac{1}{2}\frac{U^{(3)}_2}{U''_2}(X^*_s+y)|\pi^*_s+z|^2-r_s\frac{U'_2}{U''_2}(X^*_s+y)-r_sX^*_s-\pi^*_s\cdot\theta^\mathcal{H}_s-e_s+c^*_s
\end{equation*}
for $s\in[0,T]$ and $(y,z)\in\mathbb{R}\times\mathbb{R}^d$ (where we suppressed $\omega$);
\item
It holds that
\begin{equation}\label{equilibrium condition}
	\begin{cases}
		c^*_s=(U'_1)^{-1}(\lambda_2(s,T)U'_2(X^*_s+Y_s)),\\
		\pi^*_s=-\frac{U'_2}{U''_2}(X^*_s+Y_s)\theta^\mathcal{H}_s-Z^\mathcal{H}_s,
	\end{cases}
	\ \text{a.s.\ for a.e.}\,s\in[0,T].
\end{equation}
\end{enumerate}
In particular, the triplet $(X,Y,Z):=(X^*,Y,Z)$ satisfies the following coupled FBSDE:
\begin{equation}\label{FBSDE}
	\begin{cases}
		X_t=x-\int^t_0\bigl(\frac{U'_2}{U''_2}(X_s+Y_s)\theta^\mathcal{H}_s+Z^\mathcal{H}_s\bigr)\cdot dW_s\\
		\hspace{1cm}+\int^t_0\bigl(r_sX_s-|\theta^\mathcal{H}_s|^2\frac{U'_2}{U''_2}(X_s+Y_s)-\theta^\mathcal{H}_s\cdot Z^\mathcal{H}_s+e_s-(U'_1)^{-1}(\lambda_2(s,T)U'_2(X_s+Y_s))\bigr)\,ds,\\
		Y_t=E-\int^T_tZ_s\cdot dW_s\\
		\hspace{1cm}-\int^T_t\bigl(-r_sX_s+|\theta^\mathcal{H}_s|^2\frac{U'_2}{U''_2}(X_s+Y_s)+\theta^\mathcal{H}_s\cdot Z^\mathcal{H}_s-e_s+(U'_1)^{-1}(\lambda_2(s,T)U'_2(X_s+Y_s))\\
		\hspace{2cm}-r_s\frac{U'_2}{U''_2}(X_s+Y_s)-\frac{1}{2}|\theta^\mathcal{H}_s|^2\frac{U^{(3)}_2(X_s+Y_s)(U'_2(X_s+Y_s))^2}{(U''_2(X_s+Y_s))^3}-\frac{1}{2}\frac{U^{(3)}_2}{U''_2}(X_s+Y_s)|Z^\mathcal{O}_s|^2\bigr)\,ds,\\
		\hspace{13cm}t\in[0,T].
	\end{cases}
\end{equation}
\end{theo}

To prove Theorem~\ref{theorem: necessary condition}, let us show several lemmas. We fix $x\in\mathbb{R}$ and $p>1$. The method of the proof of the following lemma is inspired by Horst et al.~\cite{a_Horst-_14}.


\begin{lemm}\label{lemma: variational equality}
For any $(c^*,\pi^*)\in\Pi^{x,p}_1$, there exists a pair $(Y,Z)$ satisfying the conditions (\rnum{1}) and (\rnum{2}) in Theorem~\ref{theorem: necessary condition} such that, for any $t\in[0,T)$, $(\kappa,\eta)\in\chi_t$ and $\epsilon\in(0,T-t)$, it holds that, a.s.,
\begin{align}\label{variational equality}
	\nonumber&\frac{R(c^{t,\epsilon},\pi^{t,\epsilon};t,X^*_t)-R(c^*,\pi^*;t,X^*_t)}{\epsilon}\\
	\nonumber&=\left(\frac{1}{\epsilon}\mathbb{E}_t\left[\int^{t+\epsilon}_t\bigl(\lambda_1(t,s)U'_1(c^*_s)-\lambda_2(t,T)U'_2(X^*_s+Y_s)\bigr)\,ds\right]\right)\,\kappa\\
	\nonumber&+\lambda_2(t,T)\left(\frac{1}{\epsilon}\mathbb{E}_t\left[\int^{t+\epsilon}_t\bigl(U'_2(X^*_s+Y_s)\theta^\mathcal{H}_s+U''_2(X^*_s+Y_s)(\pi^*_s+Z^\mathcal{H}_s)\bigr)\,ds\right]\right)\cdot\eta\\
	&-\frac{1}{2\epsilon}\mathbb{E}_t\left[\int^{t+\epsilon}_t\lambda_1(t,s)\int^1_0\bigl|U''_1(c^*_s+\mu\kappa)\bigr|\,d\mu\,ds\,|\kappa|^2+\lambda_2(t,T)\int^1_0\bigl|U''_2(X^*_T+E+\mu\xi^{t,\epsilon}_T)\bigr|\,d\mu\,\bigl|\xi^{t,\epsilon}_T\bigr|^2\right].
\end{align}
\end{lemm}


\begin{proof}
Let $(c^*,\pi^*)\in\Pi^{x,p}_1$ and fix $t\in[0,T)$, $(\kappa,\eta)\in\chi_t$ and $\epsilon\in(0,T-t)$. Define $X^{t,\epsilon}:=X^{(c^{t,\epsilon},\pi^{t,\epsilon},t,X^*_t)}$ and $\xi^{t,\epsilon}:=X^{t,\epsilon}-X^*$. Then $\xi^{t,\epsilon}$ is the solution of SDE~(\ref{perturbation SDE}). Noting that $U''_1$ and $U''_2$ are negative, we see that
\begin{align}\label{variational equality'}
	\nonumber&\frac{R(c^{t,\epsilon},\pi^{t,\epsilon};t,X^*_t)-R(c^*,\pi^*;t,X^*_t)}{\epsilon}\\
	\nonumber&=\frac{1}{\epsilon}\mathbb{E}_t\left[\int^T_t\lambda_1(t,s)\bigl(U_1(c^{t,\epsilon}_s)-U_1(c^*_s)\bigr)\,ds+\lambda_2(t,T)\bigl(U_2(X^{t,\epsilon}_T+E)-U_2(X^*_T+E)\bigr)\right]\\
	\nonumber&=\left(\frac{1}{\epsilon}\mathbb{E}_t\left[\int^{t+\epsilon}_t\lambda_1(t,s)U'_1(c^*_s)\,ds\right]\right)\,\kappa+\lambda_2(t,T)\frac{1}{\epsilon}\mathbb{E}_t\bigl[U'_2(X^*_T+E)\xi^{t,\epsilon}_T\bigr]\\
	&-\frac{1}{2\epsilon}\mathbb{E}_t\left[\int^{t+\epsilon}_t\lambda_1(t,s)\int^1_0\bigl|U''_1(c^*_s+\mu\kappa)\bigr|\,d\mu\,ds\,|\kappa|^2+\lambda_2(t,T)\int^1_0\bigl|U''_2(X^*_T+E+\mu\xi^{t,\epsilon}_T)\bigr|\,d\mu\,\bigl|\xi^{t,\epsilon}_T\bigr|^2\right].
\end{align}
Now we investigate the conditional expectation $\mathbb{E}_t[U'_2(X^*_T+E)\xi^{t,\epsilon}_T]$. Note that $U'_2(X^*_T+E)\in L^p_{\mathcal{F}_T}(\Omega;\mathbb{R})$ since $(c^*,\pi^*)\in\Pi^{x,p}_1$. Let $(\alpha,\beta)\in L^p_\mathbb{F}(\Omega;C([0,T];\mathbb{R}))\times L^p_\mathbb{F}(\Omega;L^2(0,T;\mathbb{R}^d))$ be the unique adapted solution of the BSDE
\begin{equation}\label{duality BSDE}
	\begin{cases}
		d\alpha_s=-r_s\alpha_s\,ds+\beta_s\cdot dW_s,\ s\in[0,T],\\
		\alpha_T=U'_2(X^*_T+E).
	\end{cases}
\end{equation}
Then $\alpha_t=\mathbb{E}_t[e^{-\int^T_tr_s\,ds}U'_2(X^*_T+E)]$ a.s.\ for any $t\in[0,T]$; see Proposition~4.1.2 in the textbook~\cite{b_Zhang_17}. In particular, $\alpha_t$ is positive, and hence $(U'_2)^{-1}(\alpha_t)$ is well-defined for any $t\in[0,T]$ a.s. Define a process $Y$ by $Y:=(U'_2)^{-1}(\alpha)-X^*$. Then $Y$ is an $\mathbb{R}$-valued continuous adapted process such that $U'_2(X^*+Y)=\alpha\in L^p_\mathbb{F}(\Omega;C([0,T];\mathbb{R}))$ and $Y_T=E$. It\^{o}'s formula yields that
\begin{align*}
	d(X^*_s+Y_s)&=d(U'_2)^{-1}(\alpha_s)\\
	&=\frac{1}{U''_2((U'_2)^{-1}(\alpha_s))}\,d\alpha_s-\frac{1}{2}\frac{U^{(3)}_2((U'_2)^{-1}(\alpha_s))}{(U''_2((U'_2)^{-1}(\alpha_s)))^3}\,d\langle\alpha,\alpha\rangle_s\\
	&=\frac{1}{U''_2(X^*_s+Y_s)}\beta_s\cdot dW_s+\left(-\frac{1}{2}\frac{U^{(3)}_2(X^*_s+Y_s)}{(U''_2(X^*_s+Y_s))^3}|\beta_s|^2-r_s\frac{U'_2}{U''_2}(X^*_s+Y_s)\right)\,ds,
\end{align*}
and hence
\begin{align*}
	dY_s=&\left(\frac{1}{U''_2(X^*_s+Y_s)}\beta_s-\pi^*_s\right)\cdot dW_s\\
	&+\left(-\frac{1}{2}\frac{U^{(3)}_2(X^*_s+Y_s)}{(U''_2(X^*_s+Y_s))^3}|\beta_s|^2-r_s\frac{U'_2}{U''_2}(X^*_s+Y_s)-r_sX^*_s-\pi^*_s\cdot\theta^\mathcal{H}_s-e_s+c^*_s\right)\,ds.
\end{align*}
Define $Z:=\frac{1}{U''_2(X^*+Y)}\beta-\pi^*$. Then clearly $Z$ is an $\mathbb{R}^d$-valued predictable process satisfying $\int^T_0|Z_s|^2\,ds<\infty$ a.s. Moreover, from the above equation, we see that $(Y,Z)$ satisfies BSDE~(\ref{BSDE for (Y,Z)}). Hence, $(Y,Z)$ satisfies the conditions (\rnum{1}) and (\rnum{2}) in Theorem~\ref{theorem: necessary condition}. It remains to prove that the equality (\ref{variational equality}) holds. Observe that, by using It\^{o}'s formula for $(\alpha_s\xi^{t,\epsilon}_s)_{s\in[t+\epsilon,T]}$, it holds that
\begin{equation*}
	\alpha_T\xi^{t,\epsilon}_T=\alpha_{t+\epsilon}\xi^{t,\epsilon}_{t+\epsilon}+\int^T_{t+\epsilon}\xi^{t,\epsilon}_s\beta_s\cdot dW_s.
\end{equation*}
Furthermore, since $\xi^{t,\epsilon}\in\bigcap_{\gamma\geq1}L^\gamma_\mathbb{F}(\Omega;C([t,T];\mathbb{R}))$ and $\beta\in L^p_\mathbb{F}(\Omega;L^2(0,T;\mathbb{R}^d))$ with $p>1$, we have
\begin{align*}
	\mathbb{E}\Bigl[\Bigl(\int^T_t|\xi^{t,\epsilon}_s\beta_s|^2\,ds\Bigr)^{1/2}\Bigr]&\leq\mathbb{E}\Bigl[\sup_{t\leq s\leq T}|\xi^{t,\epsilon}_s|\Bigl(\int^T_t|\beta_s|^2\,ds\Bigr)^{1/2}\Bigr]\\
	&\leq \mathbb{E}\Bigl[\sup_{t\leq s\leq T}|\xi^{t,\epsilon}_s|^{p'}\Bigr]^{1/p'}\mathbb{E}\Bigl[\Bigl(\int^T_t|\beta_s|^2\,ds\Bigr)^{p/2}\Bigr]^{1/p}<\infty,
\end{align*}
where $p'>1$ is the conjugate of $p>1$. Thus the stochastic integral $\int^\cdot_t\xi^{t,\epsilon}_s\beta_s\cdot dW_s$ is a martingale, and it holds that $\mathbb{E}_{t+\epsilon}\bigl[\int^T_{t+\epsilon}\xi^{t,\epsilon}_s\beta_s\cdot dW_s\bigr]=0$ a.s. Therefore, by taking the conditional expectations, we see that
\begin{equation*}
	\mathbb{E}_t[U'_2(X^*_T+E)\xi^{t,\epsilon}_T]=\mathbb{E}_t[\alpha_T\xi^{t,\epsilon}_T]=\mathbb{E}_t\bigl[\mathbb{E}_{t+\epsilon}[\alpha_T\xi^{t,\epsilon}_T]\bigr]=\mathbb{E}_t[\alpha_{t+\epsilon}\xi^{t,\epsilon}_{t+\epsilon}]=\mathbb{E}_t[U'_2(X^*_{t+\epsilon}+Y_{t+\epsilon})\xi^{t,\epsilon}_{t+\epsilon}].
\end{equation*}
Again by It\^{o}'s formula,
\begin{align*}
	&U'_2(X^*_{t+\epsilon}+Y_{t+\epsilon})\xi^{t,\epsilon}_{t+\epsilon}\\
	&=U'_2(X^*_t+Y_t)\xi^{t,\epsilon}_t+\int^{t+\epsilon}_tU'_2(X^*_s+Y_s)\bigl(\eta\cdot dW^\mathcal{H}_s+(r_s\xi^{t,\epsilon}_s+\eta\cdot\theta^\mathcal{H}_s-\kappa)\,ds\bigr)\\
	&\hspace{0.5cm}+\int^{t+\epsilon}_t\xi^{t,\epsilon}_sU''_2(X^*_s+Y_s)\bigl((\pi^*_s+Z_s)\cdot dW_s+(r_sX^*_s+\pi^*_s\cdot\theta^\mathcal{H}_s+e_s-c^*_s+f^*(s,Y_s,Z_s))\,ds\bigr)\\
	&\hspace{0.5cm}+\frac{1}{2}\int^{t+\epsilon}_t\xi^{t,\epsilon}_sU^{(3)}_2(X^*_s+Y_s)|\pi^*_s+Z_s|^2\,ds+\int^{t+\epsilon}_tU''_2(X^*_s+Y_s)(\pi^*_s+Z^\mathcal{H}_s)\cdot\eta\,ds\displaybreak[1]\\
	&=-\left(\int^{t+\epsilon}_tU'_2(X^*_s+Y_s)\,ds\right)\,\kappa+\left(\int^{t+\epsilon}_t\bigl(U'_2(X^*_s+Y_s)\theta^\mathcal{H}_s+U''_2(X^*_s+Y_s)(\pi^*_s+Z^\mathcal{H}_s)\bigr)\,ds\right)\cdot\eta\\
	&\hspace{0.5cm}+\int^{t+\epsilon}_tU'_2(X^*_s+Y_s)\eta\cdot dW^\mathcal{H}_s+\int^{t+\epsilon}_t\xi^{t,\epsilon}_sU''_2(X^*_s+Y_s)(\pi^*_s+Z_s)\cdot dW_s.
\end{align*}
Recall that
\begin{align*}
	&U'_2(X^*+Y)=\alpha\in L^p_\mathbb{F}(\Omega;C([0,T];\mathbb{R})),\ U''_2(X^*+Y)(\pi^*+Z)=\beta\in L^p_\mathbb{F}(\Omega;L^2(0,T;\mathbb{R}^d)),\\
	&\eta\in L^\infty_{\mathcal{F}_t}(\Omega;\mathbb{R}^d),\ \xi^{t,\epsilon}\in\bigcap_{\gamma\geq1}L^\gamma_\mathbb{F}(\Omega;C([t,T];\mathbb{R})).
\end{align*}
By the same arguments as above, we can show that
\begin{equation*}
	\mathbb{E}_t\left[\int^{t+\epsilon}_tU'_2(X^*_s+Y_s)\eta\cdot dW^\mathcal{H}_s+\int^{t+\epsilon}_t\xi^{t,\epsilon}_sU''_2(X^*_s+Y_s)(\pi^*_s+Z_s)\cdot dW_s\right]=0.
\end{equation*}
Consequently, we get
\begin{align}\label{duality}
	\nonumber\mathbb{E}_t[U'_2(X^*_T+E)\xi^{t,\epsilon}_T]=&-\mathbb{E}_t\left[\int^{t+\epsilon}_tU'_2(X^*_s+Y_s)\,ds\right]\,\kappa\\
	&+\mathbb{E}_t\left[\int^{t+\epsilon}_t\bigl(U'_2(X^*_s+Y_s)\theta^\mathcal{H}_s+U''_2(X^*_s+Y_s)(\pi^*_s+Z^\mathcal{H}_s)\bigr)\,ds\right]\cdot\eta.
\end{align}
Therefore, by (\ref{variational equality'}) and (\ref{duality}), we obtain the equality (\ref{variational equality}).
\end{proof}

We want to calculate the limit of the right hand side of (\ref{variational equality}) when $\epsilon$ tends to zero. To do so, we use the following lemma which was proved by Wang~\cite{a_Wang_18}.

\begin{lemm}\label{lemma: Wang}
If $P=(P^1,\dots,P^m)^\top\in L^\gamma_\mathbb{F}(0,T;\mathbb{R}^m)$ with $m\in\mathbb{N}$ and $\gamma>1$, then for a.e.\,$t\in[0,T)$, there exists a sequence $\{\epsilon^t_n\}_{n\in\mathbb{N}}\subset(0,T-t)$ depending on $t$ such that $\lim_{n\to\infty}\epsilon^t_n=0$ and
\begin{equation*}
	\lim_{n\to\infty}\frac{1}{\epsilon^t_n}\int^{t+\epsilon^t_n}_t\mathbb{E}_t\bigl[|P^i_s-P^i_t|\bigr]\,ds=0,\ i=1,\dots,m,\ \text{a.s.}
\end{equation*}
\end{lemm}

\begin{proof}
See the proof of Lemma~3.3 in \cite{a_Wang_18}.
\end{proof}


\begin{lemm}\label{lemma: limit1}
For any $(c^*,\pi^*)\in\Pi^{x,p}_1$, consider the pair $(Y,Z)$ in Lemma~\ref{lemma: variational equality}. Then there exists a measurable set $E_1\subset[0,T)$ with $\mathrm{Leb}([0,T]\setminus E_1)=0$ such that, for any $t\in E_1$, there exists a sequence $\{\epsilon^t_n\}_{n\in\mathbb{N}}\subset(0,T-t)$ satisfying $\lim_{n\to\infty}\epsilon^t_n=0$ and, a.s.,
\begin{align*}
	&\lim_{n\to\infty}\frac{1}{\epsilon^t_n}\mathbb{E}_t\left[\int^{t+\epsilon^t_n}_t\bigl(\lambda_1(t,s)U'_1(c^*_s)-\lambda_2(t,T)U'_2(X^*_s+Y_s)\bigr)\,ds\right]\\
	&\hspace{1cm}=U'_1(c^*_t)-\lambda_2(t,T)U'_2(X^*_t+Y_t)\displaybreak[1]
	\shortintertext{and}
	&\lim_{n\to\infty}\frac{1}{\epsilon^t_n}\mathbb{E}_t\left[\int^{t+\epsilon^t_n}_t\bigl(U'_2(X^*_s+Y_s)\theta^\mathcal{H}_s+U''_2(X^*_s+Y_s)(\pi^*_s+Z^\mathcal{H}_s)\bigr)\,ds\right]\\
	&\hspace{1cm}=U'_2(X^*_t+Y_t)\theta^\mathcal{H}_t+U''_2(X^*_t+Y_t)(\pi^*_t+Z^\mathcal{H}_t).
\end{align*}
\end{lemm}


\begin{proof}
Note that, for any $t\in[0,T)$ and $\epsilon\in(0,T-t)$, it holds that
\begin{align*}
	&\mathbb{E}\left[\left|\frac{1}{\epsilon}\mathbb{E}_t\left[\int^{t+\epsilon}_t\lambda_1(t,s)U'_1(c^*_s)\,ds\right]-\frac{1}{\epsilon}\mathbb{E}_t\left[\int^{t+\epsilon}_tU'_1(c^*_s)\,ds\right]\right|\right]\\
	&\leq\max_{t\leq s\leq t+\epsilon}|\lambda_1(t,s)-1|\,\frac{1}{\epsilon}\int^{t+\epsilon}_t\mathbb{E}[U'_1(c^*_s)]\,ds.
\end{align*}
Since $\lambda_1$ is continuous and $\lambda_1(t,t)=1$, we see that $\lim_{\epsilon\downarrow0}\max_{t\leq s\leq t+\epsilon}|\lambda_1(t,s)-1|=0$ for any $t\in[0,T)$. Moreover, since $\mathbb{E}[U'_1(c^*)]\in L^p(0,T;\mathbb{R})$, by the Lebesgue differentiation theorem, there exists a measurable set $\tilde{E}_1\subset[0,T)$ with $\mathrm{Leb}([0,T]\setminus\tilde{E}_1)=0$ such that, for any $t\in\tilde{E}_1$, the term $\frac{1}{\epsilon}\int^{t+\epsilon}_t\mathbb{E}[U'_1(c^*_s)]\,ds$ converges to $\mathbb{E}[U'_1(c^*_t)]<\infty$ as $\epsilon\downarrow0$. Therefore, for each $t\in\tilde{E}_1$, it holds that
\begin{equation*}
	\lim_{\epsilon\downarrow0}\mathbb{E}\left[\left|\frac{1}{\epsilon}\mathbb{E}_t\left[\int^{t+\epsilon}_t\lambda_1(t,s)U'_1(c^*_s)\,ds\right]-\frac{1}{\epsilon}\mathbb{E}_t\left[\int^{t+\epsilon}_tU'_1(c^*_s)\,ds\right]\right|\right]=0.
\end{equation*}
Hence, for such $t$, there exists a sequence $\{\epsilon^t_n\}_{n\in\mathbb{N}}\subset(0,T-t)$ such that $\lim_{n\to\infty}\epsilon^t_n=0$ and
\begin{equation*}
	\lim_{n\to\infty}\left|\frac{1}{\epsilon^t_n}\mathbb{E}_t\left[\int^{t+\epsilon^t_n}_t\lambda_1(t,s)U'_1(c^*_s)\,ds\right]-\frac{1}{\epsilon^t_n}\mathbb{E}_t\left[\int^{t+\epsilon^t_n}_tU'_1(c^*_s)\,ds\right]\right|=0\ \text{a.s.}
\end{equation*}
Moreover, since $U'_1(c^*)\in L^p_\mathbb{F}(0,T;\mathbb{R})$, $U'_2(X^*+Y)\in L^p_\mathbb{F}(\Omega;C([0,T];\mathbb{R}))$, $\theta^\mathcal{H}$ is predictable and bounded, and $U''_2(X^*+Y)(\pi^*+Z^\mathcal{H})\in L^p_\mathbb{F}(\Omega;L^2(0,T;\mathbb{R}^d))$, Lemma~\ref{lemma: Wang} yields that there exists a measurable set $E_1\subset\tilde{E}_1$ with $\mathrm{Leb}([0,T]\setminus E_1)=0$ such that, for any $t\in E_1$, there exists a subsequence of $\{\epsilon^t_n\}_{n\in\mathbb{N}}$ which we also denote by $\{\epsilon^t_n\}_{n\in\mathbb{N}}$ such that, a.s.,
\begin{align*}
	&\lim_{n\to\infty}\frac{1}{\epsilon^t_n}\mathbb{E}_t\left[\int^{t+\epsilon^t_n}_tU'_1(c^*_s)\,ds\right]=U'_1(c^*_t),\displaybreak[1]\\
	&\lim_{n\to\infty}\frac{1}{\epsilon^t_n}\mathbb{E}_t\left[\int^{t+\epsilon^t_n}_tU'_2(X^*_s+Y_s)\,ds\right]=U'_2(X^*_t+Y_t),\displaybreak[1]\\
	&\lim_{n\to\infty}\frac{1}{\epsilon^t_n}\mathbb{E}_t\left[\int^{t+\epsilon^t_n}_tU'_2(X^*_s+Y_s)\theta^\mathcal{H}_s\,ds\right]=U'_2(X^*_t+Y_t)\theta^\mathcal{H}_t,\displaybreak[1]
	\shortintertext{and}
	&\lim_{n\to\infty}\frac{1}{\epsilon^t_n}\mathbb{E}_t\left[\int^{t+\epsilon^t_n}_tU''_2(X^*_s+Y_s)(\pi^*_s+Z^\mathcal{H}_s)\,ds\right]=U''_2(X^*_t+Y_t)(\pi^*_t+Z^\mathcal{H}_t).
\end{align*}
Hence the assertions of Lemma~\ref{lemma: limit1} follow.
\end{proof}


\begin{lemm}\label{lemma: limit2}
Let $(c^*,\pi^*)\in\Pi^{x,p}_1$, $(Y,Z)$, $E_1$ and $\{\epsilon^t_n\}_{n\in\mathbb{N}}$, $t\in E_1$, be the ones in Lemma~\ref{lemma: limit1}. Fix an arbitrary $\delta>0$. If moreover $(c^*,\pi^*)\in\Pi^{x,p}_2$, then there exists a measurable set $E_2\subset E_1$ with $\mathrm{Leb}([0,T]\setminus E_2)=0$ such that, for any $t\in E_2$ and any $(\kappa,\eta)\in\chi_t$ with $|\kappa|\leq\delta$, there exists a subsequence of $\{\epsilon^t_n\}_{n\in\mathbb{N}}$ (which we also denote by $\{\epsilon^t_n\}_{n\in\mathbb{N}}$) satisfying, a.s.,
\begin{align*}
	&\limsup_{n\to\infty}\frac{1}{2\epsilon^t_n}\mathbb{E}_t\left[\int^{t+\epsilon^t_n}_t\lambda_1(t,s)\int^1_0\bigl|U''_1(c^*_s+\mu\kappa)\bigr|\,d\mu\,ds\,|\kappa|^2\right.\\
	&\hspace{3cm}\left.+\lambda_2(t,T)\int^1_0\bigl|U''_2(X^*_T+E+\mu\xi^{t,\epsilon^t_n}_T)\bigr|\,d\mu\,\bigl|\xi^{t,\epsilon^t_n}_T\bigr|^2\right]\\
	&\leq C\bigl(M_1(c^*_t;\delta)+\mathbb{E}_t\bigl[|U''_2(X^*_T+E)|^q\bigr]^{1/q}\bigr)(|\kappa|^2+|\eta|^2).
\end{align*}
Here, $C>0$ is a constant which depends only on $q,\,T,\,\|\lambda_1\|_\infty,\,\|\lambda_2\|_\infty,\,\|r\|_\infty$ and $\|\theta^\mathcal{H}\|_\infty$, and $q$ is a constant satisfying the assertions in (H2)$_{x,p}$.
\end{lemm}


\begin{proof}
Assume that $(c^*,\pi^*)\in\Pi^{x,p}_2$ and let $q>1$ be a constant satisfying the assertions in (H2)$_{x,p}$. Fix an arbitrary $\delta>0$. Note that, for any $t\in[0,T)$, $(\kappa,\eta)\in\chi_t$ with $|\kappa|\leq\delta$ and $\epsilon\in(0,T-t)$, it holds that
\begin{align}\label{estimate of the second order}
	\nonumber&\frac{1}{2\epsilon}\mathbb{E}_t\left[\int^{t+\epsilon}_t\lambda_1(t,s)\int^1_0\bigl|U''_1(c^*_s+\mu\kappa)\bigr|\,d\mu\,ds\,|\kappa|^2+\lambda_2(t,T)\int^1_0\bigl|U''_2(X^*_T+E+\mu\xi^{t,\epsilon}_T)\bigr|\,d\mu\,\bigl|\xi^{t,\epsilon}_T\bigr|^2\right]\\
	&\leq\frac{\|\lambda_1\|_\infty}{2}\left(\frac{1}{\epsilon}\mathbb{E}_t\left[\int^{t+\epsilon}_tM_1(c^*_s;\delta)\,ds\right]\right)\,|\kappa|^2+\frac{\|\lambda_2\|_\infty}{2}\frac{1}{\epsilon}\mathbb{E}_t\Bigl[M_2\bigl(X^*_T+E;\bigl|\xi^{t,\epsilon}_T\bigr|\bigr)\,\bigl|\xi^{t,\epsilon}_T\bigr|^2\Bigr]\ \text{a.s.},
\end{align}
where $M_i$, $i=1,2$, are defined by (\ref{definition of M}). Since $M_1(c^*;\delta)\in L^q_\mathbb{F}(0,T;\mathbb{R})$, by Lemma~\ref{lemma: Wang}, there exists a measurable subset $E_2\subset E_1$ with $\mathrm{Leb}([0,T]\setminus E_2)=0$ such that, for any $t\in E_2$, there exists a subsequence of $\{\epsilon^t_n\}_{n\in\mathbb{N}}$ (which we also denote by $\{\epsilon^t_n\}_{n\in\mathbb{N}}$) such that
\begin{equation}\label{limit of M_1}
	\lim_{n\to\infty}\frac{1}{\epsilon^t_n}\mathbb{E}_t\left[\int^{t+\epsilon^t_n}_tM_1(c^*_s;\delta)\,ds\right]=M_1(c^*_t;\delta)\ \text{a.s.}
\end{equation}
Fix $t\in E_2$ and let $(\kappa,\eta)\in\chi_t$ with $|\kappa|\leq\delta$. Denote $|\xi^{t,\epsilon^t_n,\kappa,\eta}|$ by $\xi_n$ for each $n\in\mathbb{N}$. Then, by Lemma~\ref{lemma: perturbation SDE}~(\rnum{2}), it holds that
\begin{equation*}
	\mathbb{E}_t\bigl[\xi_n^{2\gamma}\bigr]\leq\bigl(C_\gamma\epsilon^t_n(|\kappa|^2+|\eta|^2)\bigr)^\gamma\ \text{a.s.}
\end{equation*}
for any $n\in\mathbb{N}$ and any $\gamma>1$, where $C_\gamma>0$ is a constant which depends only on $\gamma,\,T,\,\|r\|_\infty$ and $\|\theta^\mathcal{H}\|_\infty$. Hence, by H\"{o}lder's inequality, we obtain
\begin{align}\label{estimate of M_2}
	\nonumber\frac{1}{\epsilon^t_n}\mathbb{E}_t\Bigl[M_2\bigl(X^*_T+E;\bigl|\xi^{t,\epsilon^t_n}_T\bigr|\bigr)\,\bigl|\xi^{t,\epsilon^t_n}_T\bigr|^2\Bigr]&\leq\frac{1}{\epsilon^t_n}\mathbb{E}_t\bigl[M_2(X^*_T+E;\xi_n)^q\bigr]^{1/q}\mathbb{E}_t\bigl[\xi_n^{2q/(q-1)}\bigr]^{(q-1)/q}\\
	&\leq C_{q/(q-1)}\mathbb{E}_t\bigl[M_2(X^*_T+E;\xi_n)^q\bigr]^{1/q}(|\kappa|^2+|\eta|^2).
\end{align}
Note that, by considering a subsequence of $\{\epsilon^t_n\}_{n\in\mathbb{N}}$ (which may depend on $(\kappa,\eta)$), we can assume without loss of generality that $\lim_{n\to\infty}\xi_n=0$ a.s., and hence
\begin{equation*}
	\lim_{n\to\infty}M_2(X^*_T+E;\xi_n)=M_2(X^*_T+E;0)=|U''_2(X^*_T+E)|\ \text{a.s.}
\end{equation*}
Since the sequence $\{M_2(X^*_T+E;\xi_n)^q\}_{n\in\mathbb{N}}$ is uniformly integrable by the condition (\rnum{2}) in (H2)$_{x,p}$, we obtain
\begin{equation}\label{limit of M_2}
	\lim_{n\to\infty}\mathbb{E}_t\bigl[M_2(X^*_T+E;\xi_n)^q\bigr]=\mathbb{E}_t\bigl[|U''_2(X^*_T+E)|^q\bigr]\ \text{a.s.}
\end{equation}
By (\ref{estimate of the second order}), (\ref{limit of M_1}), (\ref{estimate of M_2}) and (\ref{limit of M_2}), we obtain the assertion of Lemma~\ref{lemma: limit2}.
\end{proof}


\begin{proof}[Proof of Theorem~\ref{theorem: necessary condition}]
Assume that $(c^*,\pi^*)\in\Pi^{x,p}_2$ is an open-loop equilibrium pair. Fix a constant $\delta>0$, and take an arbitrary $t$ from the set $E_2$ obtained in Lemma~\ref{lemma: limit2}. Let $\{\delta_m\}_{m\in\mathbb{N}}$ be a sequence such that $0<\delta_m\leq\delta$,\ $m\in\mathbb{N}$, and $\lim_{m\to\infty}\delta_m=0$. For each $m\in\mathbb{N}$, define $(\kappa_m,\eta_m)\in\chi_t$ by
\begin{equation*}
	\begin{cases}
		\kappa_m=\delta_m\1_{\{U'_1(c^*_t)-\lambda_2(t,T)U'_2(X^*_t+Y_t)>0\}},\\
		\eta_m=0.
	\end{cases}
\end{equation*}
Denote by $\{\epsilon^t_{n,m}\}_{n\in\mathbb{N}}$ the sequence corresponding to $t\in E_2$ and $(\kappa_m,\eta_m)\in\chi_t$, and let $c^{n,m}=c^{t,\epsilon^t_{n,m},\kappa_m}$, $\pi^{n,m}=\pi^{t,\epsilon^t_{n,m},\eta_m}$, for $n,\,m\in\mathbb{N}$. Then, by the definition of the open-loop equilibrium pair and Lemmas~\ref{lemma: variational equality}, \ref{lemma: limit1}, \ref{lemma: limit2}, we obtain, for each $m\in\mathbb{N}$,
\begin{align*}
	0&\geq\limsup_{n\to\infty}\frac{R(c^{n,m},\pi^{n,m};t,X^*_t)-R(c^*,\pi^*;t,X^*_t)}{\epsilon^t_{n,m}}\\
	&\geq\bigl(U'_1(c^*_t)-\lambda_2(t,T)U'_2(X^*_t+Y_t)\bigr)\1_{\{U'_1(c^*_t)-\lambda_2(t,T)U'_2(X^*_t+Y_t)>0\}}\delta_m\\
	&\hspace{1cm}-C\bigl(M_1(c^*_t;\delta)+\mathbb{E}_t\bigl[|U''_2(X^*_T+E)|^q\bigr]^{1/q}\bigr)\delta^2_m\ \text{a.s.}
\end{align*}
Dividing both sides of the above inequality by $\delta_m$ and then letting $m\to\infty$, we obtain
\begin{equation*}
	\bigl(U'_1(c^*_t)-\lambda_2(t,T)U'_2(X^*_t+Y_t)\bigr)\1_{\{U'_1(c^*_t)-\lambda_2(t,T)U'_2(X^*_t+Y_t)>0\}}\leq0\ \text{a.s.},
\end{equation*}
and hence
\begin{equation*}
	U'_1(c^*_t)-\lambda_2(t,T)U'_2(X^*_t+Y_t)\leq0\ \text{a.s.}
\end{equation*}
Next, define $(\kappa_m,\eta_m)\in\chi_t$ by
\begin{equation*}
	\begin{cases}
		\kappa_m=-\delta_m\1_{\{U'_1(c^*_t)-\lambda_2(t,T)U'_2(X^*_t+Y_t)<0\}},\\
		\eta_m=0,
	\end{cases}
\end{equation*}
for each $m\in\mathbb{N}$. By the same arguments as above, we can show that
\begin{equation*}
	U'_1(c^*_t)-\lambda_2(t,T)U'_2(X^*_t+Y_t)\geq0\ \text{a.s.}
\end{equation*}
Consequently, we obtain
\begin{equation*}
	U'_1(c^*_t)-\lambda_2(t,T)U'_2(X^*_t+Y_t)=0\ \text{a.s.},
\end{equation*}
showing that $c^*_t=(U'_1)^{-1}(\lambda_2(t,T)U'_2(X^*_t+Y_t))$ a.s.
\par
Similarly, by considering
\begin{align*}
	&\begin{cases}
		\kappa_m=0,\\
		\eta_m=\delta_m\1_{\{U'_2(X^*_t+Y_t)\theta^i_t+U''_2(X^*_t+Y_t)(\pi^{*,i}_t+Z^i_t)>0\}}e_i,
	\end{cases}
	m\in\mathbb{N},\ i=1,\dots,d_1,\displaybreak[1]
	\shortintertext{and}
	&\begin{cases}
		\kappa_m=0,\\
		\eta_m=-\delta_m\1_{\{U'_2(X^*_t+Y_t)\theta^i_t+U''_2(X^*_t+Y_t)(\pi^{*,i}_t+Z^i_t)<0\}}e_i,
	\end{cases}
	m\in\mathbb{N},\ i=1,\dots,d_1,
\end{align*}
where $e_i$ denotes the $i$\,th standard basis of $\mathbb{R}^d$ for $i=1,\dots,d$, we can show that
\begin{equation*}
	U'_2(X^*_t+Y_t)\theta^i_t+U''_2(X^*_t+Y_t)(\pi^{*,i}_t+Z^i_t)=0\ \text{a.s.},\ i=1,\dots,d_1.
\end{equation*}
Hence we get $\pi^*_t=-\frac{U'_2}{U''_2}(X^*_t+Y_t)\theta^\mathcal{H}_t-Z^\mathcal{H}_t$ a.s. Since $t\in E_2$ is arbitrary and $\mathrm{Leb}([0,T]\setminus E_2)=0$, the equalities in (\ref{equilibrium condition}) hold.
\par
Lastly, by inserting the representation~(\ref{equilibrium condition}) into SDE~(\ref{wealth SDE}) and BSDE~(\ref{BSDE for (Y,Z)}), we see that the triplet $(X,Y,Z):=(X^*,Y,Z)$ satisfies FBSDE~(\ref{FBSDE}). This completes the proof of Theorem~\ref{theorem: necessary condition}.
\end{proof}


\begin{rem}
The resulting FBSDE~(\ref{FBSDE}) is of a closed form in the sense that the coefficients of the equation are determined only by the model and independent of the choice of $(c^*,\pi^*)$. Note also that, if $d_1<d$, i.e., if the market is incomplete, then the generator of the backward equation in FBSDE~(\ref{FBSDE}) has quadratic growth with respect to $Z^\mathcal{O}=(0,\dots,0,Z^{d_1+1},\dots,Z^d)^\top$. On the other hand, If $d_1=d$, i.e., if the market is complete, then the term $|Z^\mathcal{O}_s|^2$ vanishes and all the coefficients of FBSDE~(\ref{FBSDE}) become linear with respect to $Z$.
\end{rem}


\begin{rem}\label{remark: maximum principle}
By the arguments in the proof of Lemma~\ref{lemma: variational equality}, we see that the pair of adapted processes $(Y,Z)$ corresponding to $(c^*,\pi^*)\in\Pi^{x,p}_1$ satisfies $Y=(U'_2)^{-1}(\alpha)-X^*$ and $Z=\frac{1}{U''_2(X^*+Y)}\beta-\pi^*$, where $(\alpha,\beta)\in L^p_\mathbb{F}(\Omega;C([0,T];\mathbb{R}))\times L^p_\mathbb{F}(\Omega;L^2(0,T;\mathbb{R}^d))$ is the unique adapted solution of BSDE~(\ref{duality BSDE}). Therefore, if $(c^*,\pi^*)\in\Pi^{x,p}_2$ is an open-loop equilibrium pair, then by the condition~(\ref{equilibrium condition}) we obtain
\begin{equation}\label{equilibrium condition'}
	\begin{cases}
		U'_1(c^*_s)-\lambda_2(s,T)\alpha_s=0,\\
		\alpha_s\theta^\mathcal{H}_s+\beta^\mathcal{H}_s=0,
	\end{cases}
	\text{a.s. for a.e.}\,s\in[0,T].
\end{equation}
Moreover, if we define $(p,q):\Omega\times\Delta[0,T]\to\mathbb{R}\times\mathbb{R}^d$ by $p^t_s:=\lambda_2(t,T)\alpha_s$ and $q^t_s:=\lambda_2(t,T)\beta_s$ for $(t,s)\in\Delta[0,T]$, then, for each $t\in[0,T]$, $(p^t,q^t)\in L^p_\mathbb{F}(\Omega;C([t,T];\mathbb{R}))\times L^p_\mathbb{F}(\Omega;L^2(t,T;\mathbb{R}^d))$ satisfies the BSDE
\begin{equation}\label{flow of adjoint equations}
	\begin{cases}
		dp^t_s=-r_sp^t_s\,ds+q^t_s\cdot dW_s,\ s\in[t,T],\\
		p^t_T=\lambda_2(t,T)U'_2(X^*_T+E),
	\end{cases}
\end{equation}
and the condition (\ref{equilibrium condition'}) becomes
\begin{equation}\label{equilibrium condition''}
	\begin{cases}
		U'_1(c^*_s)-p^s_s=0,\\
		p^s_s\theta^\mathcal{H}_s+(q^s_s)^\mathcal{H}=0,
	\end{cases}
	\text{a.s. for a.e.}\,s\in[0,T].
\end{equation}
This consequence generalizes the result obtained by Alia et al.~\cite{a_Alia-_17} to an incomplete market setting. The random field $(p,q)$ satisfies a ``flow'' of BSDEs (\ref{flow of adjoint equations}) with the additional condition (\ref{equilibrium condition''}) imposed on the diagonal terms $(p^s_s,q^s_s)$. Note that the flow of BSDEs (\ref{flow of adjoint equations}) is a system of adjoint equations parametrized by $t\in[0,T]$ in the spirit of the duality method for time-inconsistent stochastic control problems; see \cite{a_Hu-Jin-Zhou_12,a_Hu-Jin-Zhou_17,a_Alia-_17,a_Yong_17,a_Yan-Yong_19}. However, unlike FBSDE~(\ref{FBSDE}), we cannot obtain an equation of a closed form directly by the system consisting of (\ref{wealth SDE}), (\ref{flow of adjoint equations}) and (\ref{equilibrium condition''}), since the condition~(\ref{equilibrium condition''}) does not give an expression for the investment process $\pi^*$. Hence, in our problem, the necessary condition stated in Theorem~\ref{theorem: necessary condition} has more information for the structure of an open-loop equilibrium pair than the consequence obtained by the duality method.
\end{rem}


\section{A verification theorem}\label{section: verification theorem}

\paragraph{}
\ \,In Section \ref{section: necessary condition}, we proved that, if there exists an open-loop equilibrium pair in the set $\Pi^{x,p}_2$, then FBSDE~(\ref{FBSDE}) has an adapted solution. In this section, we prove the inverse direction. In other words, we provide a verification theorem for an open-loop equilibrium pair. Throughout this section, we fix $U_1,U_2\in\mathbb{U}$ and $\lambda_1,\lambda_2\in\Lambda$.


\begin{theo}\label{theorem: verification}
Fix $x\in\mathbb{R}$ and $p>1$. Suppose that there exists a triplet $(X,Y,Z)$ such that:
\begin{enumerate}
\renewcommand{\labelenumi}{(\roman{enumi})}
\item
$X$ and $Y$ are $\mathbb{R}$-valued continuous adapted processes, and $Z$ is an $\mathbb{R}^d$-valued predictable process satisfying $\int^T_0|Z_s|^2\,ds<\infty$ a.s.;
\item
$U'_2(X+Y)\in L^p_\mathbb{F}(\Omega;C([0,T];\mathbb{R}))$ and
\begin{equation*}
	\mathbb{E}\left[\int^T_0\bigl|U_1\bigl((U'_1)^{-1}(\lambda_2(s,T)U'_2(X_s+Y_s))\bigr)\bigr|\,ds+|U_2(X_T+E)|\right]<\infty;
\end{equation*}
\item
$(X,Y,Z)$ satisfies FBSDE (\ref{FBSDE}).
\end{enumerate}
Define $(c^*,\pi^*)$ by
\begin{equation}\label{equilibrium condition2}
	\begin{cases}
		c^*_s=(U'_1)^{-1}(\lambda_2(s,T)U'_2(X_s+Y_s)),\\
		\pi^*_s=-\frac{U'_2}{U''_2}(X_s+Y_s)\theta^\mathcal{H}_s-Z^\mathcal{H}_s,
	\end{cases}
	\ s\in[0,T].
\end{equation}
Then $(c^*,\pi^*)\in\Pi^{x,p}_1$, and it is an open-loop equilibrium pair for the initial wealth $x$ with the corresponding wealth process $X^{(c^*,\pi^*,0,x)}=X$.
\end{theo}


\begin{proof}
Clearly the pair of predictable processes $(c^*,\pi^*)$ defined by (\ref{equilibrium condition2}) satisfies $\int^T_0(|c^*_s|+|\pi^*_s|^2)\,ds<\infty$ a.s., and it holds that $\pi^*=(\pi^{*,1},\dots,\pi^{*,d_1},0,\dots,0)^\top$. The continuous process $X$ satisfies SDE~(\ref{wealth SDE}) with the initial value $x$ at time zero and $(c,\pi)=(c^*,\pi^*)$. By the uniqueness of the continuous solution of SDE~(\ref{wealth SDE}), we get $X^*:=X^{(c^*,\pi^*,0,x)}=X$. Then, by the assumption (\rnum{2}) of this theorem, we see that $(c^*,\pi^*)\in\Pi^x_0$. Moreover, since
\begin{align*}
	\mathbb{E}\left[\int^T_0U'_1(c^*_s)^p\,ds+U'_2(X^*_T+E)^p\right]&=\mathbb{E}\left[\int^T_0\bigl(\lambda_2(s,T)U'_2(X_s+Y_s)\bigr)^p\,ds+U'_2(X_T+Y_T)^p\right]\\
	&\leq\left(\int^T_0\lambda_2(s,T)^p\,ds+1\right)\mathbb{E}\left[\sup_{0\leq s\leq T}U'_2(X_s+Y_s)^p\right]<\infty,
\end{align*}
we see that $(c^*,\pi^*)\in\Pi^{x,p}_1$.
\par
Now let us show that $(c^*,\pi^*)$ is an open-loop equilibrium pair. Fix arbitrary $t\in[0,T)$, $(\kappa,\eta)\in\chi_t$ and $\epsilon\in(0,T-t)$. Define $X^{t,\epsilon}:=X^{(c^{t,\epsilon},\pi^{t,\epsilon},t,X^*_t)}$ and $\xi^{t,\epsilon}:=X^{t,\epsilon}-X^*$. Then $\xi^{t,\epsilon}$ is the solution of SDE~(\ref{perturbation SDE}). Since $U_1$ and $U_2$ are concave, we have
\begin{align}\label{estimate R}
	\nonumber&\frac{R(c^{t,\epsilon},\pi^{t,\epsilon};t,X^*_t)-R(c^*,\pi^*;t,X^*_t)}{\epsilon}\\
	\nonumber&\leq\frac{1}{\epsilon}\mathbb{E}_t\left[\int^{t+\epsilon}_t\lambda_1(t,s)U'_1(c^*_s)\,ds\,\kappa+\lambda_2(t,T)U'_2(X^*_T+E)(X^{t,\epsilon}_T-X^*_T)\right]\\
	&=\frac{1}{\epsilon}\mathbb{E}_t\left[\int^{t+\epsilon}_t\lambda_1(t,s)\lambda_2(s,T)U'_2(X_s+Y_s)\,ds\,\kappa+\lambda_2(t,T)U'_2(X_T+Y_T)\xi^{t,\epsilon}_T\right].
\end{align}
Define $\alpha:=U'_2(X+Y)$. By the assumption, $\alpha$ is in $L^p_\mathbb{F}(\Omega;C([0,T];\mathbb{R}))$. Moreover, by It\^{o}'s formula,
\begin{align*}
	d\alpha_s&=U''_2(X_s+Y_s)(dX_s+dY_s)+\frac{1}{2}U^{(3)}_2(X_s+Y_s)\,d\langle X+Y,X+Y\rangle_s\\
	&=U''_2(X_s+Y_s)\Biggl(r_sX_s-|\theta^\mathcal{H}_s|^2\frac{U'_2}{U''_2}(X_s+Y_s)-\theta^\mathcal{H}_s\cdot Z^\mathcal{H}_s+e_s-(U'_1)^{-1}(\lambda_2(s,T)U'_2(X_s+Y_s))\\
	&\hspace{0.7cm}-r_sX_s+|\theta^\mathcal{H}_s|^2\frac{U'_2}{U''_2}(X_s+Y_s)+\theta^\mathcal{H}_s\cdot Z^\mathcal{H}_s-e_s+(U'_1)^{-1}(\lambda_2(s,T)U'_2(X_s+Y_s))\\
	&\hspace{0.7cm}-r_s\frac{U'_2}{U''_2}(X_s+Y_s)-\frac{1}{2}|\theta^\mathcal{H}_s|^2\frac{U^{(3)}_2(X_s+Y_s)(U'_2(X_s+Y_s))^2}{(U''_2(X_s+Y_s))^3}-\frac{1}{2}\frac{U^{(3)}_2}{U''_2}(X_s+Y_s)|Z^\mathcal{O}_s|^2\Biggr)\,ds\\
	&\hspace{0.4cm}+\frac{1}{2}U^{(3)}_2(X_s+Y_s)\left|-\Bigl(\frac{U'_2}{U''_2}(X_s+Y_s)\theta^\mathcal{H}_s+Z^\mathcal{H}_s\Bigr)+Z_s\right|^2\,ds\\
	&\hspace{0.4cm}+U''_2(X_s+Y_s)\left(-\Bigl(\frac{U'_2}{U''_2}(X_s+Y_s)\theta^\mathcal{H}_s+Z^\mathcal{H}_s\Bigr)+Z_s\right)\cdot dW_s\displaybreak[1]\\
	&=-r_sU'_2(X_s+Y_s)\,ds-U'_2(X_s+Y_s)\theta^\mathcal{H}_s\cdot dW^\mathcal{H}_s+U''_2(X_s+Y_s)Z^\mathcal{O}_s\cdot dW^\mathcal{O}_s.
\end{align*}
Define $\beta:=-U'_2(X+Y)\theta^\mathcal{H}+U''_2(X+Y)Z^\mathcal{O}$. Then it holds that $\int^T_0|\beta_s|^2\,ds<\infty$ a.s., and $(\alpha,\beta)$ satisfies the BSDE
\begin{equation*}
	\begin{cases}
		d\alpha_s=-r_s\alpha_s\,ds+\beta_s\cdot dW_s,\ s\in[0,T],\\
		\alpha_T=U'_2(X_T+Y_T).
	\end{cases}
\end{equation*}
Since $\alpha\in L^p_\mathbb{F}(\Omega;C([0,T];\mathbb{R}))$ and $r$ is bounded, by using the Burkholder--Davis--Gundy inequality, we can easily show that $\beta\in L^p_\mathbb{F}(\Omega;L^2(0,T;\mathbb{R}^d))$. Then, by the same arguments as in the proof of Lemma~\ref{lemma: variational equality}, we obtain
\begin{equation*}
	\mathbb{E}_t[U'_2(X_T+Y_T)\xi^{t,\epsilon}_T]=\mathbb{E}_t[\alpha_T\xi^{t,\epsilon}_T]=\mathbb{E}_t\bigl[\mathbb{E}_{t+\epsilon}[\alpha_T\xi^{t,\epsilon}_T]\bigr]=\mathbb{E}_t[\alpha_{t+\epsilon}\xi^{t,\epsilon}_{t+\epsilon}].
\end{equation*}
Furthermore, by applying It\^{o}'s formula to $(\alpha_s\xi^{t,\epsilon}_s)_{s\in[t,t+\epsilon]}$, we have
\begin{align*}
	&\alpha_{t+\epsilon}\xi^{t,\epsilon}_{t+\epsilon}\\
	&=\alpha_t\xi^{t,\epsilon}_t+\int^{t+\epsilon}_t\alpha_s\bigl((r_s\xi^{t,\epsilon}_s+\eta\cdot\theta^\mathcal{H}_s-\kappa)\,ds+\eta\cdot dW^\mathcal{H}_s\bigr)\\
	&\hspace{2cm}+\int^{t+\epsilon}_t\xi^{t,\epsilon}_s(-r_s\alpha_s\,ds+\beta_s\cdot dW_s)+\int^{t+\epsilon}_t\beta_s\cdot\eta\,ds\displaybreak[1]\\
	&=\left(\int^{t+\epsilon}_t(\alpha_s\theta^\mathcal{H}_s+\beta^\mathcal{H}_s)\,ds\right)\cdot\eta-\left(\int^{t+\epsilon}_t\alpha_s\,ds\right)\kappa+\int^{t+\epsilon}_t\alpha_s\eta\cdot dW^\mathcal{H}_s+\int^{t+\epsilon}_t\xi^{t,\epsilon}_s\beta_s\cdot dW_s.
\end{align*}
Note that, by the definition of $(\alpha,\beta)$, it holds that $\alpha\theta^\mathcal{H}+\beta^\mathcal{H}=0$. Moreover, since
\begin{align*}
	&\alpha\in L^p_\mathbb{F}(\Omega;C([0,T];\mathbb{R})),\ \beta\in L^p_\mathbb{F}(\Omega;L^2(0,T;\mathbb{R}^d)),\\
	&\eta\in L^\infty_{\mathcal{F}_t}(\Omega;\mathbb{R}^d),\ \xi^{t,\epsilon}\in\bigcap_{\gamma\geq1}L^\gamma_\mathbb{F}(\Omega;C([t,T];\mathbb{R})),
\end{align*}
we can show that
\begin{equation*}
	\mathbb{E}_t\left[\int^{t+\epsilon}_t\alpha_s\eta\cdot dW^\mathcal{H}_s+\int^{t+\epsilon}_t\xi^{t,\epsilon}_s\beta_s\cdot dW_s\right]=0\ \text{a.s.}
\end{equation*}
Consequently, we get
\begin{equation*}
	\mathbb{E}_t[U'_2(X_T+Y_T)\xi^{t,\epsilon}_T]=-\mathbb{E}_t\left[\int^{t+\epsilon}_t\alpha_s\,ds\right]\,\kappa.
\end{equation*}
From this equality and the estimate (\ref{estimate R}), we obtain
\begin{align*}
	&\frac{R(c^{t,\epsilon},\pi^{t,\epsilon};t,X^*_t)-R(c^*,\pi^*;t,X^*_t)}{\epsilon}\\
	&\leq\frac{1}{\epsilon}\mathbb{E}_t\left[\int^{t+\epsilon}_t\bigl(\lambda_1(t,s)\lambda_2(s,T)-\lambda_2(t,T)\bigr)\alpha_s\,ds\right]\,\kappa\\
	&\leq\max_{t\leq s\leq t+\epsilon}|\lambda_1(t,s)\lambda_2(s,T)-\lambda_2(t,T)|\ \mathbb{E}_t\left[\sup_{t\leq s\leq T}|\alpha_s|\right]\,|\kappa|\ \text{a.s.}
\end{align*}
Since $\lambda_1$ and $\lambda_2$ are continuous, noting that $\lambda_1(t,t)=1$, the third line of the above inequalities tends to zero as $\epsilon\downarrow0$. Consequently, we obtain
\begin{equation*}
	\limsup_{\epsilon\downarrow0}\frac{R(c^{t,\epsilon},\pi^{t,\epsilon};t,X^*_t)-R(c^*,\pi^*;t,X^*_t)}{\epsilon}\leq0\ \text{a.s.},
\end{equation*}
and hence $(c^*,\pi^*)$ is an open-loop equilibrium pair.
\end{proof}


\begin{rem}\label{remark: lambda_1}
Note that FBSDE~(\ref{FBSDE}) does not depend on the discount function $\lambda_1$ of the utility of consumptions. Consequently, if the conditions in Theorem~\ref{theorem: verification} are satisfied, then the pair $(c^*,\pi^*)$ defined by (\ref{equilibrium condition2}) is an open-loop equilibrium pair of the problem defined by $(\lambda_1,\lambda_2,U_1,U_2)$ for any $\lambda_1\in\Lambda$.
\end{rem}


\begin{rem}\label{remark: BMO}
Let us remark on the $L^p$-integrability assumption in Theorem~\ref{theorem: verification}. If $(X,Y,Z)$ satisfies conditions (\rnum{1}) and (\rnum{3}) in the theorem, then the process $U'_2(X+Y)$ evolves as
\begin{align*}
	U'_2(X_t+Y_t)=&U'_2(x+Y_0)-\int^t_0r_sU'_2(X_s+Y_s)\,ds\\
	&-\int^t_0U'_2(X_s+Y_s)\theta^\mathcal{H}_s\cdot dW^\mathcal{H}_s+\int^t_0U''_2(X_s+Y_s)Z^\mathcal{O}_s\cdot dW^\mathcal{O}_s,\ t\in[0,T].
\end{align*}
Therefore, by letting $\nu:=-\theta^\mathcal{H}+\frac{U''_2}{U'_2}(X+Y)Z^\mathcal{O}$, we get
\begin{equation*}
	U'_2(X+Y)=U'_2(x+Y_0)\exp\Bigl(-\int^\cdot_0r_s\,ds\Bigr)\mathcal{E}\Bigl(\int^\cdot_0\nu_s\cdot dW_s\Bigr),
\end{equation*}
where $\mathcal{E}(\cdot)$ denotes the stochastic exponent. If $\int^\cdot_0\nu_s\cdot dW_s$ is a BMO martingale, then by Theorem~7.2.3 of the textbook~\cite{b_Zhang_17}, there exists a constant $p>1$ such that
\begin{equation*}
	\mathcal{E}\Bigl(\int^\cdot_0\nu_s\cdot dW_s\Bigr)\in L^p_\mathbb{F}(\Omega;C([0,T];\mathbb{R})),
\end{equation*}
and hence the process $U'_2(X+Y)$ is also in $L^p_\mathbb{F}(\Omega;C([0,T];\mathbb{R}))$.
\end{rem}

Now we investigate the well-posedness of FBSDE~\eqref{FBSDE} in the case of exponential utility functions. We assume that the interest rate process $(r_s)_{s\in[0,T]}$ is a constant $r\geq0$. Fix arbitrary discount functions $\lambda_1,\,\lambda_2\in\Lambda$. For $i=1,2$, let $U_i(x)=-\exp(-\gamma_i x)$ for some constant $\gamma_i>0$. Then we have
\begin{equation*}
	\frac{U'_2(x)}{U''_2(x)}=-\frac{1}{\gamma_2},\ \frac{U^{(3)}_2(x)(U'_2(x))^2}{(U''_2(x))^3}=-\frac{1}{\gamma_2},\ \frac{U^{(3)}_2(x)}{U''_2(x)}=-\gamma_2,
\end{equation*}
and
\begin{equation*}
	(U'_1)^{-1}(\lambda_2(s,T)U'_2(x))=\frac{\gamma_2}{\gamma_1}x-\frac{1}{\gamma_1}\log\left(\frac{\gamma_2}{\gamma_1}\lambda_2(s,T)\right)
\end{equation*}
for $x\in\mathbb{R}$ and $s\in[0,T]$. Thus, FBSDE~\eqref{FBSDE} becomes the following:
\begin{equation}\label{expFBSDE}
	\begin{cases}
		X_t=x+\int^t_0\bigl(\frac{1}{\gamma_2}\theta^\mathcal{H}_s-Z^\mathcal{H}_s\bigr)\cdot dW_s\\
		\hspace{1cm}+\int^t_0\bigl(rX_s-\frac{\gamma_2}{\gamma_1}(X_s+Y_s)-\theta^\mathcal{H}_s\cdot Z^\mathcal{H}_s+\frac{1}{\gamma_2}|\theta^\mathcal{H}_s|^2+e_s+\frac{1}{\gamma_1}\log(\frac{\gamma_2}{\gamma_1}\lambda_2(s,T))\bigr)\,ds,\\
		Y_t=E-\int^T_tZ_s\cdot dW_s\\
		\hspace{1cm}-\int^T_t\bigl(-rX_s+\frac{\gamma_2}{\gamma_1}(X_s+Y_s)+\theta^\mathcal{H}_s\cdot Z^\mathcal{H}_s+\frac{\gamma_2}{2}|Z^\mathcal{O}_s|^2\\
		\hspace{3cm}-\frac{1}{2\gamma_2}|\theta^\mathcal{H}_s|^2-e_s-\frac{1}{\gamma_1}\log(\frac{\gamma_2}{\gamma_1}\lambda_2(s,T))+\frac{r}{\gamma_2}\bigr)\,ds,\ \ t\in[0,T].
	\end{cases}
\end{equation}
We transform $(X,Y,Z)$ by
\begin{equation*}
	\tilde{X}_t:=h(t)X_t,\ \tilde{Y}_t:=Y_t+(1-h(t))X_t\ \text{and}\ \tilde{Z}_t:=Z_t+(1-h(t))\left(\frac{1}{\gamma_2}\theta^\mathcal{H}_t-Z^\mathcal{H}_t\right),
\end{equation*}
where $h(t)$ is a positive deterministic function defined by
\begin{equation*}
	h(t):=
	\begin{cases}
	\frac{1}{1+\frac{\gamma_2}{\gamma_1}(T-t)}\ \text{if}\ r=0,\\
	\frac{r}{\frac{\gamma_2}{\gamma_1}-(\frac{\gamma_2}{\gamma_1}-r)\exp(-r(T-t))}\ \text{if}\ r>0
	\end{cases}
\end{equation*}
Note that $h$ satisfies the following differential equation:
\begin{equation*}
	\dot{h}(t)=h(t)\left(\frac{\gamma_2}{\gamma_1}h(t)-r\right),\ t\in[0,T],\ h(T)=1.
\end{equation*}
By using It\^{o}'s formula, we see that FBSDE~\eqref{expFBSDE} is transformed to the following:
\begin{equation}\label{expFBSDE'}
	\begin{cases}
		\tilde{X}_t=h(0)x+\int^t_0\bigl(\frac{1}{\gamma_2}\theta^\mathcal{H}_s-\tilde{Z}^\mathcal{H}_s\bigr)\cdot dW_s\\
		\hspace{1cm}+\int^t_0\bigl(-\frac{\gamma_2}{\gamma_1}h(s)\tilde{Y}_s-\theta^\mathcal{H}_s\cdot \tilde{Z}^\mathcal{H}_s+\frac{1}{\gamma_2}|\theta^\mathcal{H}_s|^2+h(s)e_s+\frac{h(s)}{\gamma_1}\log(\frac{\gamma_2}{\gamma_1}\lambda_2(s,T))\bigr)\,ds,\\
		\tilde{Y}_t=E-\int^T_t\tilde{Z}_s\cdot dW_s\\
		\hspace{1cm}-\int^T_t\bigl(\frac{\gamma_2}{\gamma_1}h(s)\tilde{Y}_s+\theta^\mathcal{H}_s\cdot \tilde{Z}^\mathcal{H}_s+\frac{\gamma_2}{2}|\tilde{Z}^\mathcal{O}_s|^2\\
		\hspace{3cm}-\frac{1}{2\gamma_2}|\theta^\mathcal{H}_s|^2-h(s)e_s-\frac{h(s)}{\gamma_1}\log(\frac{\gamma_2}{\gamma_1}\lambda_2(s,T))+\frac{r}{\gamma_2}\bigr)\,ds,\\
		\hspace{13cm}t\in[0,T].
	\end{cases}
\end{equation}
By the construction, if $(X,Y,Z)$ solves FBSDE~\eqref{expFBSDE}, then $(\tilde{X},\tilde{Y},\tilde{Z})$ solves the system~\eqref{expFBSDE'}. Conversely, if $(\tilde{X},\tilde{Y},\tilde{Z})$ satisfies \eqref{expFBSDE'}, then $(X,Y,Z)$ defined by
\begin{equation}\label{transform}
	X_t:=\frac{1}{h(t)}\tilde{X}_t,\ Y_t:=\tilde{Y}_t+\left(1-\frac{1}{h(t)}\right)\tilde{X}_t\ \text{and}\ Z_t:=\tilde{Z}_t+\left(1-\frac{1}{h(t)}\right)\left(\frac{1}{\gamma_2}\theta^\mathcal{H}_t-\tilde{Z}^\mathcal{H}_t\right)
\end{equation}
solves FBSDE~\eqref{expFBSDE}. We note that the system~\eqref{expFBSDE'} is a decoupled FBSDE, and the generator of the BSDE for $(\tilde{Y},\tilde{Z})$ has quadratic growth with respect to $\tilde{Z}^\mathcal{O}$. From this observation we obtain the following result.


\begin{prop}
Assume that the interest rate process $(r_s)_{s\in[0,T]}$ is a constant $r\geq0$, and both the rate of income process $(e_s)_{s\in[0,T]}$ and the $\mathcal{F}_T$-measurable lump-sum payment $E$ are bounded. Fix arbitrary discount functions $\lambda_1,\,\lambda_2\in\Lambda$. For $i=1,2$, let $U_i(x)=-\exp(-\gamma_i x)$ for some constant $\gamma_i>0$. Then there exists a constant $p>1$ such that, for any $x\in\mathbb{R}$, there exists a triplet $(X,Y,Z)$ satisfying the conditions (\rnum{1}), (\rnum{2}) and (\rnum{3}) in Theorem~\ref{theorem: verification}. In particular, $(c^*,\pi^*)$ defined by \eqref{equilibrium condition2} is in $\Pi^{x,p}_2$, and it is an open-loop equilibrium pair for the initial wealth $x$ with the corresponding wealth process $X^{(c^*,\pi^*,0,x)}=X$.
\end{prop}


\begin{proof}
By Lemma~\ref{lemma: utility functions}, we have $\Pi^{x,p}_1=\Pi^{x,p}_2$ for any $x\in\mathbb{R}$ and $p>1$. Thus, it suffices to show that there exists a triplet $(X,Y,Z)$ satisfying the conditions (\rnum{1}), (\rnum{2}) and (\rnum{3}) in Theorem~\ref{theorem: verification} for some $p>1$. It follows from Theorems~7.2.1 and 7.3.3 in the textbook~\cite{b_Zhang_17} that the BSDE in the system~\eqref{expFBSDE'} has a unique solution $(\tilde{Y},\tilde{Z})$ such that $\tilde{Y}$ is bounded and $\int^\cdot_0\tilde{Z}_s\cdot dW_s$ is a BMO martingale. For each $x\in\mathbb{R}$, define $\tilde{X}$ by the first equation in the system~\eqref{expFBSDE'}, and define $(X,Y,Z)$ by \eqref{transform}. Then $(X,Y,Z)$ satisfies the conditions (\rnum{1}) and (\rnum{3}) in Theorem~\ref{theorem: verification}. It remains to show that the condition (\rnum{2}) holds for some $p>1$. Noting that
\begin{equation*}
	\nu:=-\theta^\mathcal{H}+\frac{U''_2}{U'_2}(X+Y)Z^\mathcal{O}=-\theta^\mathcal{H}-\gamma_2\tilde{Z}^\mathcal{O},
\end{equation*}
we see that $\int^\cdot_0\nu_s\cdot dW_s$ is a BMO martingale. Thus, by the discussions in Remark~\ref{remark: BMO}, there exists a constant $p>1$ (which does not depend on $x$) such that
\begin{equation*}
	U'_2(X+Y)=\gamma_2\exp(-\gamma_2(X+Y))\in L^p_\mathbb{F}(\Omega;C([0,T];\mathbb{R})).
\end{equation*}
Moreover, we have
\begin{equation*}
	\mathbb{E}\bigl[|U_2(X_T+E)|\bigr]=\mathbb{E}\bigl[\exp(-\gamma_2(X_T+Y_T))\bigr]<\infty
\end{equation*}
and
\begin{equation*}
\begin{split}
	&\mathbb{E}\left[\int^T_0\bigl|U_1\bigl((U'_1)^{-1}(\lambda_2(s,T)U'_2(X_s+Y_s))\bigr)\bigr|\,ds\right]\\
	&=\mathbb{E}\left[\int^T_0\frac{\gamma_2}{\gamma_1}\lambda_2(s,T)\exp(-\gamma_2(X_s+Y_s))\,ds\right]<\infty.
\end{split}
\end{equation*}
Hence, the condition (\rnum{2}) in Theorem~\ref{theorem: verification} holds and we complete the proof.
\end{proof}


\begin{rem}
The quadratic BSDE in the system~\eqref{expFBSDE'} has the same structure as the one in Cheridito and Hu~\cite{a_cheridito-Hu_11}. For the general FBSDE~\eqref{FBSDE}, we mention the study of Fromm and Imkeller~\cite{a_Fromm-Imkeller_20}. They showed the well-posedness of a similar FBSDE (without the terms $r_sX_s$ and $(U'_1)^{-1}(\lambda_2(s,T)U'_2(X_s+Y_s))$) by using the method of decoupling fields. However, we cannot apply their results directly to FBSDE~\eqref{FBSDE} due to some technical difficulties stem from the appearance of the above two terms. Thus, FBSDE~\eqref{FBSDE} of the general form is beyond the literature, and the well-posedness is an open-problem at this moment.
\end{rem}


\section{An equivalent time-consistent problem}\label{section: equivalent problem}

\paragraph{}
\ \,In this section, we investigate a relationship between an open-loop equilibrium pair of a time-inconsistent control problem and an optimal pair of a time-consistent problem. Recall that the reward functional $R$ is determined by $(\lambda_1,\lambda_2,U_1,U_2)$ for each $\lambda_1,\lambda_2\in\Lambda$ and $U_1,U_2\in\mathbb{U}$. We refer to the time-inconsistent consumption-investment problem with coefficients $(\lambda_1,\lambda_2,U_1,U_2)$ and an initial wealth $x\in\mathbb{R}$ as Problem~(I)$^x_{\lambda_1,\lambda_2,U_1,U_2}$.
\par
By Theorems~\ref{theorem: necessary condition},\ \ref{theorem: verification} and Remark~\ref{remark: lambda_1}, we easily obtain the following corollary.


\begin{cor}\label{corollary: lambda_1}
Suppose that a triplet $(\lambda_2,U_1,U_2)\in\Lambda\times\mathbb{U}\times\mathbb{U}$ is given. Fix an initial wealth $x\in\mathbb{R}$ and let $(c^*,\pi^*)\in\Pi^{x,p}_2$ for some $p>1$. Then the following are equivalent:
\begin{enumerate}
\renewcommand{\labelenumi}{(\roman{enumi})}
\item
For some $\lambda_1\in\Lambda$, $(c^*,\pi^*)$ is an open-loop equilibrium pair for Problem~(I)$^x_{\lambda_1,\lambda_2,U_1,U_2}$;
\item
For any $\lambda_1\in\Lambda$, $(c^*,\pi^*)$ is an open-loop equilibrium pair for Problem~(I)$^x_{\lambda_1,\lambda_2,U_1,U_2}$.
\end{enumerate}
\end{cor}

Next, we introduce a time-consistent consumption-investment problem, which turns out to be equivalent to Problem~(I)$^x_{\lambda_1,\lambda_2,U_1,U_2}$ in some sense. For each triplet $(\lambda_2,U_1,U_2)\in\Lambda\times\mathbb{U}\times\mathbb{U}$, define
\begin{equation*}
	\mathcal{R}(c,\pi;x):=\mathbb{E}\left[\int^T_0\frac{1}{\lambda_2(s,T)}U_1(c_s)\,ds+U_2(X^{(c,\pi,0,x)}_T+E)\right]
\end{equation*}
for $x\in\mathbb{R}$ and $(c,\pi)\in\Pi^x_0$. Moreover, we consider the set
\begin{equation*}
	\Pi^{x,p}_3:=\Pi^x_0\cap\bigl(L^p_\mathbb{F}(\Omega;L^1(0,T;\mathbb{R}))\times L^p_\mathbb{F}(\Omega;L^2(0,T;\mathbb{R}^d))\bigr)
\end{equation*}
for each $x\in\mathbb{R}$ and $p>1$. A standard utility maximization problem for the reward functional $\mathcal{R}$ (within the class of open-loop controls) is stated as follows: For each $x\in\mathbb{R}$ and $p>1$, find a pair $(c^*,\pi^*)\in\Pi^{x,p}_3$ such that
\begin{equation}\label{maximization}
	\mathcal{R}(c^*,\pi^*;x)=\sup_{(c,\pi)\in\Pi^{x,p}_3}\mathcal{R}(c,\pi;x).
\end{equation}
We refer to the above maximization problem as Problem~(C)$^{x,p}_{\lambda_2,U_1,U_2}$. We call a pair $(c^*,\pi^*)$ satisfying (\ref{maximization}) an optimal pair. It is well known that Problem~(C)$^{x,p}_{\lambda_2,U_1,U_2}$ is time-consistent. Indeed, the following lemma holds true.


\begin{lemm}\label{lemma: time-consistency}
Let $(\lambda_2,U_1,U_2)\in\Lambda\times\mathbb{U}\times\mathbb{U}$ and an initial wealth $x\in\mathbb{R}$ be given. Suppose that $(c^*,\pi^*)\in\Pi^{x,p}_3$ for some $p>1$, and denote the corresponding wealth process by $X^*$. If $(c^*,\pi^*)$ is an optimal pair for Problem~(C)$^{x,p}_{\lambda_2,U_1,U_2}$, then for any $t\in[0,T)$ and $(c,\pi)\in\Pi^{x,p}_3$ satisfying $(c_s,\pi_s)=(c^*_s,\pi^*_s)$ for $s\in[0,t)$, it holds that
\begin{equation*}
	\mathcal{R}(c,\pi;t,X^*_t)\leq\mathcal{R}(c^*,\pi^*;t,X^*_t)\ \text{a.s.},
\end{equation*}
where
\begin{equation*}
	\mathcal{R}(c,\pi;t,X^*_t):=\mathbb{E}_t\left[\int^T_t\frac{1}{\lambda_2(s,T)}U_1(c_s)\,ds+U_2(X^{(c,\pi,t,X^*_t)}_T+E)\right].
\end{equation*}
\end{lemm}


\begin{proof}
See Appendix~\ref{appendix}.
\end{proof}

The next theorem is the main result of this section.


\begin{theo}\label{theorem: equivalent problems}
Fix coefficients $(\lambda_1,\lambda_2,U_1,U_2)\in\Lambda\times\Lambda\times\mathbb{U}\times\mathbb{U}$ and an initial wealth $x\in\mathbb{R}$. Let $(c^*,\pi^*)\in\Pi^{x,p}_2\cap\Pi^{x,p/(p-1)}_3$ for some $p>1$. Then the following are equivalent:
\begin{enumerate}
\renewcommand{\labelenumi}{(\roman{enumi})}
\item
$(c^*,\pi^*)$ is an open-loop equilibrium pair for Problem~(I)$^x_{\lambda_1,\lambda_2,U_1,U_2}$;
\item
$(c^*,\pi^*)$ is an optimal pair for Problem~(C)$^{x,p/(p-1)}_{\lambda_2,U_1,U_2}$.
\end{enumerate}
In particular, among all consumption-investment pairs in the set $\Pi^{x,p}_2\cap\Pi^{x,p/(p-1)}_3$, the open-loop equilibrium pair for Problem~(I)$^x_{\lambda_1,\lambda_2,U_1,U_2}$ is, if it exists, unique.
\end{theo}


\begin{proof}
Firstly, assume that $(c^*,\pi^*)$ is an optimal pair for Problem~(C)$^{x,p/(p-1)}_{\lambda_2,U_1,U_2}$. Fix arbitrary $t\in[0,T)$, $(\kappa,\eta)\in\chi_t$ and $\epsilon\in(0,T-t)$. We extend the pair $(c^{t,\epsilon},\pi^{t,\epsilon})$ to $[0,T]$ by defining $(c^{t,\epsilon}_s,\pi^{t,\epsilon}_s):=(c^*_s,\pi^*_s)$ for $s\in[0,t)$. Then $X^{t,\epsilon}:=X^{(c^{t,\epsilon},\pi^{t,\epsilon},t,X^*_t)}$ is also extended as $X^{t,\epsilon}_s=X^*_s=X^{(c^{t,\epsilon},\pi^{t,\epsilon},0,x)}_s$ for $s\in[0,t)$. Note that
\begin{align*}
	&|U_1(c^{t,\epsilon}_s)|\leq|U_1(c^*_s)|+U'_1(c^*_s)\,|\kappa|\1_{[t,t+\epsilon)}(s)+\frac{1}{2}M_1(c^*_s;|\kappa|)\,|\kappa|^2\1_{[t,t+\epsilon)}(s),\ s\in[0,T],
	\shortintertext{and}
	&|U_2(X^{t,\epsilon}_T+E)|\leq|U_2(X^*_T+E)|+U'_2(X^*_T+E)\,|\xi^{t,\epsilon}_T|+\frac{1}{2}M_2(X^*_T+E;|\xi^{t,\epsilon}_T|)\,|\xi^{t,\epsilon}_T|^2,
\end{align*}
where $M_1$ and $M_2$ are defined by (\ref{definition of M}), and $\xi^{t,\epsilon}:=X^{t,\epsilon}-X^*$, which satisfies SDE~(\ref{perturbation SDE}). By these inequalities and the assumption that $(c^*,\pi^*)$ is in $\Pi^{x,p}_2$ with $p>1$, we can show that
\begin{equation*}
	\mathbb{E}\left[\int^T_0|U_1(c^{t,\epsilon}_s)|\,ds+|U_2(X^{t,\epsilon}_T+E)|\right]<\infty,
\end{equation*}
proving $(c^{t,\epsilon},\pi^{t,\epsilon})\in\Pi^{x}_0$. Clearly $(c^{t,\epsilon},\pi^{t,\epsilon})\in L^{p/(p-1)}_\mathbb{F}(\Omega;L^1(0,T;\mathbb{R}))\times L^{p/(p-1)}_\mathbb{F}(\Omega;L^2(0,T;\mathbb{R}^d))$ since $(c^*,\pi^*)\in\Pi^{x,p/(p-1)}_3$. Hence, we have that $(c^{t,\epsilon},\pi^{t,\epsilon})\in\Pi^{x,p/(p-1)}_3$. Since $(c^*,\pi^*)$ is an optimal pair for Problem~(C)$^{x,p/(p-1)}_{\lambda_2,U_1,U_2}$, by Lemma~\ref{lemma: time-consistency}, we obtain
\begin{align*}
	0&\geq\mathcal{R}(c^{t,\epsilon},\pi^{t,\epsilon};t,X^*_t)-\mathcal{R}(c^*,\pi^*;t,X^*_t)\\
	&=\mathbb{E}_t\left[\int^T_t\frac{1}{\lambda_2(s,T)}\bigl(U_1(c^{t,\epsilon}_s)-U_1(c^*_s)\bigr)\,ds+\bigl(U_2(X^{t,\epsilon}_T+E)-U_2(X^*_T+E)\bigr)\right]\\
	&=\frac{1}{\lambda_2(t,T)}\mathbb{E}_t\left[\int^T_t\frac{\lambda_2(t,T)}{\lambda_2(s,T)}\bigl(U_1(c^{t,\epsilon}_s)-U_1(c^*_s)\bigr)\,ds+\lambda_2(t,T)\bigl(U_2(X^{t,\epsilon}_T+E)-U_2(X^*_T+E)\bigr)\right]\ \text{a.s.}
\end{align*}
Define a function $\tilde{\lambda}_1:\Delta[0,T]\to\mathbb{R}_+$ by
\begin{equation*}
	\tilde{\lambda}_1(t,s):=\frac{\lambda_2(t,T)}{\lambda_2(s,T)},\ (t,s)\in\Delta[0,T].
\end{equation*}
Clearly $\tilde{\lambda}_1$ is in $\Lambda$. Denote by $\tilde{R}$ the reward functional corresponding to $(\tilde{\lambda}_1,\lambda_2,U_1,U_2)$. Then we get $\tilde{R}(c^{t,\epsilon},\pi^{t,\epsilon};t,X^*_t)-\tilde{R}(c^*,\pi^*;t,X^*_t)\leq0$ a.s., in particular,
\begin{equation*}
	\limsup_{\epsilon\downarrow0}\frac{\tilde{R}(c^{t,\epsilon},\pi^{t,\epsilon};t,X^*_t)-\tilde{R}(c^*,\pi^*;t,X^*_t)}{\epsilon}\leq0\ \text{a.s.},
\end{equation*}
and hence $(c^*,\pi^*)$ is an open-loop equilibrium pair for Problem~(I)$^x_{\tilde{\lambda}_1,\lambda_2,U_1,U_2}$. By Corollary~\ref{corollary: lambda_1}, we see that $(c^*,\pi^*)$ is an open-loop equilibrium pair for Problem~(I)$^x_{\lambda_1,\lambda_2,U_1,U_2}$ too.
\par
Conversely, assume that $(c^*,\pi^*)$ is an open-loop equilibrium pair for  Problem~(I)$^x_{\lambda_1,\lambda_2,U_1,U_2}$. By Theorem~\ref{theorem: necessary condition}, there exists a pair $(Y,Z)$ such that the assertions in that theorem hold true. Define $(\alpha,\beta)$ by
\begin{equation*}
	\alpha:=U'_2(X^*+Y)\ \text{and}\ \beta:=-U'_2(X^*+Y)\theta^\mathcal{H}+U''_2(X^*+Y)Z^\mathcal{O}.
\end{equation*}
Then, by the same arguments as in the proof of Theorem~\ref{theorem: verification}, we can show that
 $(\alpha,\beta)$ is in $L^p_\mathbb{F}(\Omega;C([0,T];\mathbb{R}))\times L^p_\mathbb{F}(\Omega;L^2(0,T;\mathbb{R}^d))$, and it satisfies BSDE~(\ref{duality BSDE}). Now take an arbitrary $(c,\pi)\in\Pi^{x,p/(p-1)}_3$. Define $(\hat{c},\hat{\pi}):=(c-c^*,\pi-\pi^*)$ and $\hat{\xi}:=X^{(c,\pi,0,x)}-X^*$. Then $\hat{\xi}$ satisfies the SDE
\begin{equation*}
	\begin{cases}
		d\hat{\xi}_s=(r_s\hat{\xi}_s+\hat{\pi}_s\cdot\theta^\mathcal{H}_s-\hat{c}_s)\,ds+\hat{\pi}_s\cdot dW^\mathcal{H}_s,\ s\in[0,T],\\
		\hat{\xi}_0=0.
	\end{cases}
\end{equation*}
Since $r$ and $\theta$ are bounded and $(\hat{c},\hat{\pi})\in  L^{p/(p-1)}_\mathbb{F}(\Omega;L^1(0,T;\mathbb{R}))\times L^{p/(p-1)}_\mathbb{F}(\Omega;L^2(0,T;\mathbb{R}^d))$, it can be easily shown that $\hat{\xi}\in L^{p/(p-1)}_\mathbb{F}(\Omega;C([0,T];\mathbb{R}))$; see for example Theorem~3.4.3 in the textbook~\cite{b_Zhang_17}. By the concavity of $U_1$ and $U_2$, we have
\begin{equation}\label{concavity}
	\mathcal{R}(c,\pi;x)-\mathcal{R}(c^*,\pi^*;x)\leq\mathbb{E}\left[\int^T_0\frac{1}{\lambda_2(s,T)}U'_1(c^*_s)\hat{c}_s\,ds+U'_2(X^*_T+E)\hat{\xi}_T\right].
\end{equation}
Since $c^*$ has the representation in (\ref{equilibrium condition}), we get
\begin{equation}\label{c^* equality}
	\frac{1}{\lambda_2(s,T)}U'_1(c^*_s)=U'_2(X^*_s+Y_s)=\alpha_s\ \text{a.s. for a.e.}\,s\in[0,T].
\end{equation}
Moreover, It\^{o}'s formula yields that
\begin{align*}
	&U'_2(X^*_T+E)\hat{\xi}_T=\alpha_T\hat{\xi}_T\\
	&=\int^T_0\alpha_s\bigl(\hat{\pi}_s\cdot dW^\mathcal{H}_s+(r_s\hat{\xi}_s+\hat{\pi}_s\cdot\theta^\mathcal{H}_s-\hat{c}_s)\,ds\bigr)+\int^T_0\hat{\xi}_s\bigl(\beta_s\cdot dW_s-r_s\alpha_s\,ds\bigr)+\int^T_0\hat{\pi}_s\cdot\beta_s\,ds\displaybreak[1]\\
	&=\int^T_0(\alpha_s\theta^\mathcal{H}_s+\beta^\mathcal{H}_s)\cdot\hat{\pi}_s\,ds-\int^T_0\alpha_s\hat{c}_s\,ds+\int^T_0\alpha_s\hat{\pi}_s\cdot dW^\mathcal{H}_s+\int^T_0\hat{\xi}_s\beta_s\cdot dW_s\\
	&=-\int^T_0\alpha_s\hat{c}_s\,ds+\int^T_0\alpha_s\hat{\pi}_s\cdot dW^\mathcal{H}_s+\int^T_0\hat{\xi}_s\beta_s\cdot dW_s,
\end{align*}
where in the last equality we used the relation $\alpha\theta^\mathcal{H}+\beta^\mathcal{H}=0$. Note that, by H\"{o}lder's inequality,
\begin{align*}
	\mathbb{E}\left[\left(\int^T_0|\alpha_s\hat{\pi}_s|^2\,ds\right)^{1/2}\right]&\leq\mathbb{E}\left[\sup_{0\leq s\leq T}|\alpha_s|\left(\int^T_0|\hat{\pi}_s|^2\,ds\right)^{1/2}\right]\\
	&\leq\mathbb{E}\left[\sup_{0\leq s\leq T}|\alpha_s|^p\right]^{1/p}\mathbb{E}\left[\left(\int^T_0|\hat{\pi}_s|^2\,ds\right)^{p/(2(p-1))}\right]^{(p-1)/p}<\infty.
\end{align*}
Similarly we get $\mathbb{E}[(\int^T_0|\hat{\xi}_s\beta_s|^2\,ds)^{1/2}]<\infty$. Therefore, we see that
\begin{equation*}
	\mathbb{E}\left[\int^T_0\alpha_s\hat{\pi}_s\cdot dW^\mathcal{H}_s+\int^T_0\hat{\xi}_s\beta_s\cdot dW_s\right]=0,
\end{equation*}
and hence
\begin{equation}\label{U'_2 equality}
	\mathbb{E}\bigl[U'_2(X^*_T+E)\hat{\xi}_T\bigr]=-\mathbb{E}\left[\int^T_0\alpha_s\hat{c}_s\,ds\right].
\end{equation}
By (\ref{concavity}),\,(\ref{c^* equality}) and (\ref{U'_2 equality}), we obtain $\mathcal{R}(c,\pi;x)\leq\mathcal{R}(c^*,\pi^*;x)$. Hence, $(c^*,\pi^*)$ is an optimal pair for Problem~(C)$^{x,p/(p-1)}_{\lambda_2,U_1,U_2}$.
\par
Lastly, since $U_1$ and $U_2$ are strictly concave, we see that the optimal pair for Problem~(C)$^{x,p/(p-1)}_{\lambda_2,U_1,U_2}$ is, if it exists, unique. Therefore, the last assertion of the theorem holds true. This completes the proof.
\end{proof}

Note that, if $\lambda_1$ and $\lambda_2$ are exponential discount functions, i.e., $\lambda_1(t,s)=\lambda_2(t,s)=e^{-\delta(s-t)}$, $(t,s)\in\Delta[0,T]$, with a constant $\delta\geq0$, then for any $(c,\pi)\in\Pi^x_0$ it holds that
\begin{align*}
	\mathcal{R}(c,\pi;x)&=\mathbb{E}\left[\int^T_0\frac{1}{e^{-\delta(T-s)}}U_1(c_s)\,ds+U_2(X^{(c,\pi)}_T+E)\right]\\
	&=e^{\delta T}\mathbb{E}\left[\int^T_0e^{-\delta s}U_1(c_s)\,ds+e^{-\delta T}U_2(X^{(c,\pi)}_T+E)\right]\displaybreak[1]\\
	&=e^{\delta T}\mathbb{E}\left[\int^T_0\lambda_1(0,s)U_1(c_s)\,ds+\lambda_2(0,T)U_2(X^{(c,\pi)}_T+E)\right]\\
	&=e^{\delta T}R(c,\pi;0,x).
\end{align*}
Similarly, for any $t\in[0,T)$ and any $(c,\pi)\in\Pi^x_0$ satisfying $(c_s,\pi_s)=(c^*_s,\pi^*_s)$, $s\in[0,t)$, for a fixed $(c^*,\pi^*)\in\Pi^x_0$, it holds that
\begin{equation*}
	\mathcal{R}(c,\pi;t,X^*_t)=e^{\delta(T-t)}R(c,\pi;t,X^*_t)\ \text{a.s.}
\end{equation*}
Therefore, we obtain the following corollary, which would be expected (but not trivial from the definition).


\begin{cor}
Let $(U_1,U_2)\in\mathbb{U}\times\mathbb{U}$ be given, and assume that $\lambda_1(t,s)=\lambda_2(t,s)=e^{-\delta(s-t)}$, $(t,s)\in\Delta[0,T]$, for a constant $\delta\geq0$. Fix an initial wealth $x\in\mathbb{R}$ and let $(c^*,\pi^*)\in\Pi^{x,p}_2\cap\Pi^{x,p/(p-1)}_3$ for some $p>1$. Then the following are equivalent:
\begin{enumerate}
\renewcommand{\labelenumi}{(\roman{enumi})}
\item
$(c^*,\pi^*)$ is an open-loop equilibrium pair;
\item
$(c^*,\pi^*)$ is optimal in $\Pi^{x,p/(p-1)}_3$ when viewed at the initial time, i.e., the following holds:
\begin{equation*}
	R(c^*,\pi^*;0,x)=\sup_{(c,\pi)\in\Pi^{x,p/(p-1)}_3}R(c,\pi;0,x).
\end{equation*}
\end{enumerate}
Moreover, if the above conditions hold, then for any $t\in[0,T)$ and any $(c,\pi)\in\Pi^{x,p}_3$ with $(c_s,\pi_s)=(c^*_s,\pi^*_s)$, $s\in[0,t)$, it holds that
\begin{equation*}
	R(c,\pi;t,X^*_t)\leq R(c^*,\pi^*;t,X^*_t)\ \text{a.s.}
\end{equation*}
\end{cor}


\section*{Acknowledgments}
The author would like to thank Professor Jiongmin Yong for his help and relevant discussions as well as his hospitality during the author's stay in Florida. The author would like to thank the editor and the referees for their constructive comments and suggestions. This work was supported by JSPS KAKENHI Grant Number JP18J20973.


\bibliography{reference}


\appendix
\section{Appendix}\label{appendix}

\paragraph{}
\ \,In this appendix, we provide proofs of Lemmas~\ref{lemma: perturbation SDE}, \ref{lemma: utility functions}, and \ref{lemma: time-consistency}.


\begin{proof}[Proof of Lemma~\ref{lemma: perturbation SDE}]
The assertion (\rnum{1}) is well-known; see for example Theorem~3.4.3 in \cite{b_Zhang_17}. We prove the assertion~(\rnum{2}). Let $\epsilon\in(0,T-t)$ and a set $A\in\mathcal{F}_t$ be fixed. Then $\xi^{t,\epsilon}\1_A$ satisfies the following SDE:
\begin{equation*}
	\begin{cases}
		d(\xi^{t,\epsilon}_s\1_A)=\eta\1_{[t,t+\epsilon)}(s)\1_A\cdot dW^\mathcal{H}_s+\bigl(r_s\xi^{t,\epsilon}_s\1_A+(\eta\cdot\theta^\mathcal{H}_s-\kappa)\1_{[t,t+\epsilon)}(s)\1_A\bigr)\,ds,\ s\in[t,T],\\
		\xi^{t,\epsilon}_t\1_A=0.
	\end{cases}
\end{equation*}
Again by Theorem 3.4.3 in \cite{b_Zhang_17}, for any $\gamma\geq1$, there exists a constant $C=C(\gamma,T,\|r\|_\infty)>0$ such that
\begin{align*}
	\mathbb{E}\Bigl[\sup_{t\leq s\leq T}|\xi^{t,\epsilon}_s|^{2\gamma}\1_A\Bigr]&\leq C\mathbb{E}\Bigl[\Bigl(\int^T_t|\eta\cdot\theta^\mathcal{H}_s-\kappa|\1_{[t,t+\epsilon)}(s)\1_A\,ds\Bigr)^{2\gamma}+\Bigl(\int^T_t|\eta|^2\1_{[t,t+\epsilon)}(s)\1_A\,ds\Bigr)^\gamma\Bigr]\\
	&\leq C\mathbb{E}\Bigl[\Bigl(\bigl(\|\theta^\mathcal{H}\|_\infty|\eta|+|\kappa|\bigr)^{2\gamma}\epsilon^{2\gamma}+|\eta|^{2\gamma}\epsilon^\gamma\Bigr)\1_A\Bigr]\\
	&\leq\mathbb{E}\Bigl[\bigl(C_\gamma\epsilon(|\kappa|^2+|\eta|^2)\bigr)^\gamma\1_A\Bigr],
\end{align*}
where $C_\gamma>0$ is a constant which depends only on $\gamma,\,T,\,\|r\|_\infty$ and $\|\theta^\mathcal{H}\|_\infty$. Since $A\in\mathcal{F}_t$ is arbitrary, we obtain the assertion (\rnum{2}). Next, we prove the assertion~(\rnum{3}). Recall the dynamics~\eqref{S^0 dynamics} of $S^0$. It can be easily shown that $S^0_s=\exp(\int^s_0r_u\,du)$, and hence both $S^0$ and $\frac{1}{S^0}$ are uniformly bounded. By It\^{o}'s formula, we see that the process $\bigl(\frac{\xi^{t,\epsilon}_s}{S^0_s}\bigr)_{s\in[t,T]}$ satisfies
\begin{equation*}
	\begin{cases}
		d\Bigl(\frac{\xi^{t,\epsilon}_s}{S^0_s}\Bigr)=\frac{1}{S^0_s}\eta\1_{[t,t+\epsilon)}(s)\cdot dW^\mathcal{H}_s+\frac{1}{S^0_s}(\eta\cdot\theta^\mathcal{H}_s-\kappa)\1_{[t,t+\epsilon)}(s)\,ds,\ s\in[t,T],\\
		\frac{\xi^{t,\epsilon}_t}{S^0_t}=0.
	\end{cases}
\end{equation*}
Thus we get
\begin{equation*}
	\xi^{t,\epsilon}_T=S^0_T\int^{t+\epsilon}_t\frac{1}{S^0_s}\eta\cdot dW^\mathcal{H}_s+S^0_T\int^{t+\epsilon}_t\frac{1}{S^0_s}(\eta\cdot\theta^\mathcal{H}_s-\kappa)\,ds.
\end{equation*}
Therefore, it suffices to prove that $\sup_{\epsilon\in(0,T-t)}\mathbb{E}\bigl[\exp\bigl(c\bigl|\int^{t+\epsilon}_t\frac{1}{S^0_s}\eta\cdot dW^\mathcal{H}_s\bigr|\bigr)\bigr]<\infty$. More generally, for any $\mathbb{R}^d$-valued predictable and bounded process $\varphi$, it holds that
\begin{equation}\label{exp phi estimate}
	\mathbb{E}\Bigl[\exp\Bigl(c\Bigl|\int^{t+\epsilon}_t\varphi_s\cdot dW_s\Bigr|\Bigr)\Bigr]\leq2\exp\Bigl(\frac{c^2}{2}\|\varphi\|^2_\infty\epsilon\Bigr).
\end{equation}
Indeed, we have
\begin{align*}
	&\mathbb{E}\Bigl[\exp\Bigl(c\Bigl|\int^{t+\epsilon}_t\varphi_s\cdot dW_s\Bigr|\Bigr)\Bigr]\\
	&\leq\mathbb{E}\Bigl[\exp\Bigl(c\int^{t+\epsilon}_t\varphi_s\cdot dW_s\Bigr)\Bigr]+\mathbb{E}\Bigl[\exp\Bigl(-c\int^{t+\epsilon}_t\varphi_s\cdot dW_s\Bigr)\Bigr]\displaybreak[1]\\
	&\leq\exp\Bigl(\frac{c^2}{2}\|\varphi\|^2_\infty\epsilon\Bigr)\Bigl\{\mathbb{E}\Bigl[\exp\Bigl(c\int^{t+\epsilon}_t\varphi_s\cdot dW_s-\frac{c^2}{2}\int^{t+\epsilon}_t|\varphi_s|^2\,ds\Bigr)\Bigr]\\
	&\hspace{4cm}+\mathbb{E}\Bigl[\exp\Bigl(-c\int^{t+\epsilon}_t\varphi_s\cdot dW_s-\frac{c^2}{2}\int^{t+\epsilon}_t|\varphi_s|^2\,ds\Bigr)\Bigr]\Bigr\}.
\end{align*}
Since the last two expectations are equal to $1$, we obtain the estimate~\eqref{exp phi estimate}. This completes the proof of the assertion~(\rnum{3}).
\end{proof}


\begin{proof}[Proof of Lemma~\ref{lemma: utility functions}]
Note that, for $i=1,2$, $U'_i$ is positive and decreasing. Hence, for any $x\in\mathbb{R}$ and $\delta\geq0$, it holds that
\begin{equation*}
	M_i(x;\delta):=\max_{y\in\mathbb{R},\,|y|\leq\delta}|U''_i(x+y)|=\max_{y\in\mathbb{R},\,|y|\leq\delta}\left|\frac{U''_i(x+y)}{U'_i(x+y)}U'_i(x+y)\right|\leq\left\|\frac{U''_i}{U'_i}\right\|_\infty U'_i(x-\delta).
\end{equation*}
Moreover, the assumption (\rnum{2}) of this lemma yields that $U'_i(x-\delta)\leq\exp(K\delta)U'_i(x)$, and hence we obtain $M_i(x,\delta)\leq\|\frac{U''_i}{U'_i}\|_\infty\exp(K\delta)U'_i(x)$. Let $(c,\pi)\in\Pi^{x,p}_1$ with $x\in\mathbb{R}$ and $p>1$ be given. Then, for any $\delta\geq0$, it holds that
\begin{equation*}
	\mathbb{E}\left[\int^T_0M_1(c_s;\delta)^p\,ds\right]\leq\left\|\frac{U''_1}{U'_1}\right\|^p_\infty\exp(pK\delta)\mathbb{E}\left[\int^T_0U'_1(c_s)^p\,ds\right]<\infty.
\end{equation*}
Therefore, the same is true for any $q\leq p$. Furthermore, H\"{o}lder's inequality yields that, for any $q\in(1,p)$,
\begin{align*}
	&\sup_{\epsilon\in(0,T-t)}\mathbb{E}[M_2(X^{(c,\pi)}_T+E;|\xi^{t,\epsilon}_T|)^q]\\
	&\leq\left\|\frac{U''_2}{U'_2}\right\|^q_\infty\sup_{\epsilon\in(0,T-t)}\mathbb{E}\left[\exp(qK|\xi^{t,\epsilon}_T|)U'_2(X^{(c,\pi)}_T+E)^q\right]\\
	&\leq\left\|\frac{U''_2}{U'_2}\right\|^q_\infty\sup_{\epsilon\in(0,T-t)}\mathbb{E}\left[\exp\left(\frac{pq}{p-q}K|\xi^{t,\epsilon}_T|\right)\right]^{(p-q)/p}\mathbb{E}\left[U'_2(X^{(c,\pi)}_T+E)^p\right]^{q/p}\\
	&<\infty,
\end{align*}
where we used Lemma~\ref{lemma: perturbation SDE}~(\rnum{3}) in the last estimate. This implies that the family of random variables $\{M_2(X^{(c,\pi)}_T+E;|\xi^{t,\epsilon}_T|)^q\}_{\epsilon\in(0,T-t)}$ is uniformly integrable for any $q\in(1,p)$. Hence $(c,\pi)\in\Pi^{x,p}_2$ and this completes the proof.
\end{proof}


\begin{proof}[Proof of Lemma~\ref{lemma: time-consistency}]
Fix arbitrary $t\in[0,T)$ and $(c,\pi)\in\Pi^{x,p}_3$ satisfying $(c_s,\pi_s)=(c^*_s,\pi^*_s)$ for $s\in[0,t)$. Let $A\in\mathcal{F}_t$ be an arbitrary set and consider the pair $(c^{t,A},\pi^{t,A})\in\Pi^{x,p}_3$ defined by
\begin{equation*}
	c^{t,A}_s:=
	\begin{cases}
		c^*_s\ &\text{for}\ s\in[0,t),\\
		\1_Ac_s+\1_{A^\text{c}}c^*_s\ &\text{for}\ s\in[t,T],
	\end{cases}\ \ 
	\pi^{t,A}_s:=
	\begin{cases}
		\pi^*_s\ &\text{for}\ s\in[0,t),\\
		\1_A\pi_s+\1_{A^\text{c}}\pi^*_s\ &\text{for}\ s\in[t,T].
	\end{cases}
\end{equation*}
Denote the corresponding wealth process $X^{(c^{t,A},\pi^{t,A},0,x)}$ by $X^{t,A}$. Then clearly
\begin{equation*}
	X^{t,A}_s=
	\begin{cases}
		X^*_s\ &\text{for}\ s\in[0,t),\\
		\1_AX^{(c,\pi,t,X^*_t)}_s+\1_{A^\text{c}}X^*_s\ &\text{for}\ s\in[t,T].
	\end{cases}
\end{equation*}
Since $(c^*,\pi^*)$ is an optimal pair for Problem~(C)$^{x,p}_{\lambda_2,U_1,U_2}$, we get
\begin{align*}
	0&\geq\mathcal{R}(c^{t,A},\pi^{t,A};x)-\mathcal{R}(c^*,\pi^*;x)\\
	&=\mathbb{E}\left[\int^T_0\frac{1}{\lambda_2(s,T)}\bigl(U_1(c^{t,A}_s)-U_1(c^*_s)\bigr)\,ds+\bigl(U_2(X^{t,A}_T+E)-U_2(X^*_T+E)\bigr)\right]\\
	&=\mathbb{E}\left[\left(\int^T_t\frac{1}{\lambda_2(s,T)}\bigl(U_1(c_s)-U_1(c^*_s)\bigr)\,ds+\bigl(U_2(X^{(c,\pi,t,X^*_t)}_T+E)-U_2(X^*_T+E)\bigr)\right)\1_A\right]\\
	&=\mathbb{E}\left[\bigl(\mathcal{R}(c,\pi;t,X^*_t)-\mathcal{R}(c^*,\pi^*;t,X^*_t)\bigr)\1_A\right].
\end{align*}
Since $A\in\mathcal{F}_t$ is arbitrary, we obtain the result.
\end{proof}


\end{document}